\documentclass{article}

\usepackage[english]{babel}
\usepackage[ruled,vlined,linesnumbered]{algorithm2e} 
\usepackage{graphicx}
\usepackage{subcaption}
\usepackage{multirow,multicol}
\usepackage{amsthm}
 \usepackage[table]{xcolor} 
\usepackage[letterpaper,top=2cm,bottom=2cm,left=3cm,right=3cm,marginparwidth=1.75cm]{geometry}
\newtheorem{theorem}{Theorem}
\newtheorem{remark}{Remark}

\newtheorem{lemma}{Lemma}
\newtheorem{example}{Example}

\newtheorem{assumption}{Assumption}
\usepackage{amsmath}
\usepackage{booktabs}
\usepackage{mathtools}
\usepackage{amssymb}
\usepackage{bm}
\usepackage{graphicx}
\usepackage[colorlinks=true, allcolors=blue]{hyperref}
\usepackage{authblk}

\title{An alternating learning-based collocation method for solving inverse elliptic problems}
\author[1]{Zhizhong Kong\thanks{zhizhongkong@whu.edu.cn.}}
\author[1]{Jerry Zhijian Yang\thanks{zjyang.math@whu.edu.cn.}}
\author[2]{Cheng Yuan\thanks{yuancheng@whu.edu.cn.}}
\affil[1]{School of Mathematics and Statistics, Wuhan University, Wuhan, 430072, Hubei, China}
\affil[2]{School of Artificial Intelligence, Wuhan University, Wuhan, 430072, Hubei, China}
\date{}

\begin{document}
\maketitle

\begin{abstract}
We propose the Alternating Learning-Based Collocation (ALBC) method for solving inverse elliptic problems. Our approach employs sinusoidal shallow networks as adaptive basis generators. By alternately updating the state variable and the unknown parameter, we decompose the original nonconvex joint optimization problem into a sequence of tractable linear subproblems. This strategy effectively overcomes the fixed-basis limitations of classical collocation methods while avoiding the slow convergence typically encountered in deep learning approaches. Theoretically, we establish stability estimates and prove the convergence of the proposed algorithm. Numerical experiments on five benchmark problems demonstrate the efficacy of ALBC, which consistently outperforms the standard collocation method in accuracy. Furthermore, it achieves performance comparable to or better than that of physics-informed neural networks at a substantially lower computational cost. Finally, the method remains robust under noise levels of up to twenty percent.
\end{abstract}

\section{Introduction}

Inverse problems, which aim to infer unknown model parameters from indirect observational data, play a central role in scientific and engineering disciplines such as seismic wave inversion \cite{bui2013computational,he2021reparameterized,tarantola1984inversion}, biomedical imaging \cite{anastasio2007application,arridge1999optical,bardsley2020mcmc}, environmental monitoring \cite{atmadja2001state,iglesias2013ensemble,moghaddam2021inverse} and data assimilation \cite{ding2024nonlinear}. Classical numerical strategies for solving such problems rely on discretizing the underlying PDE and embedding it within an optimization or regularization framework. Mesh-based methods, such as the finite element method (FEM) \cite{bangerth2008framework} and spectral collocation \cite{norgren2005chebyshev}, provide mature theoretical foundations with rigorous convergence guarantees, but require explicit mesh generation and repeated forward PDE solves, resulting in substantial computational overhead that grows prohibitively with problem dimension. Meshless approaches, including radial basis function (RBF) collocation \cite{li2015local}, circumvent mesh generation and extend more naturally to higher dimensions, yet remain critically dependent on the choice of basis functions: an ill-suited basis can lead to spectral mismatch, ill-conditioned algebraic systems, or poor convergence of the reconstruction. Although regularization strategies such as Tikhonov regularization \cite{benning2018modern,tikhonov1963solution} and Total Variation minimization \cite{rudin1992nonlinear} can stabilize the inversion, the overall reconstruction quality remains fundamentally constrained by the expressiveness of the chosen basis. More critically, all these classical approaches require a priori specification of the approximation basis whose regularity must match the unknown solution, yet offer limited adaptivity for resolving localized features. This central limitation has motivated the exploration of neural network-based methods, which construct data-driven representations and exhibit inherent robustness to observational noise.

In recent years, owing to their significant robustness to observational noise~\cite{jin2022imaging,mishra2023estimates,zhang2023stability}, neural networks have been widely applied to the inverse problems of partial differential equations~\cite{bao2020numerical,ji2025potential,zhang2023solving} to address these challenges. Physics-Informed Machine Learning (PIML) stands out as a representative approach~\cite{gao2022physics,pang2019fpinns,raissi2019physics}, which treats the unknown coefficients as trainable parameters while enforcing physical laws, enabling simultaneous recovery of hidden parameters and solutions. For the reconstruction of non-constant coefficients in PDEs, a prevalent strategy involves approximating both the solution and the unknown coefficient function using two separate neural networks. These networks are coupled via a composite loss function that integrates a least-squares data-fitting term with a physics-informed constraint. This methodology has since been extensively adapted to a diverse array of inverse problems~\cite{duan2024current,duan2024recovering,lu2021physics}.

While PIML methods have demonstrated impressive numerical performance in prior studies, they face significant challenges: their theoretical analysis remains under-developed, and they often suffer from high computational overhead, typically requiring a vast number of training epochs to achieve high-precision solutions. Conversely, classical numerical methods benefit from rigorous theoretical foundations. Recent studies, such as \cite{weng2026deep}, have combined the collocation method with neural networks, effectively leveraging the strengths of both approaches to achieve excellent results in PDE problems with rigorous theoretical guarantees. Similarly, \cite{xiong2025deep} integrated the Finite Element Method with neural networks, yielding comparable success. These findings highlight the substantial potential of hybridizing traditional numerical methods with neural networks.

Motivated by this paradigm, we propose a hybrid framework that integrates the collocation method with neural networks to solve inverse problems. Fundamentally, we parameterize the basis functions of the collocation method using neural networks, optimizing their parameters via an alternating learning strategy. Compared with classical numerical approaches, our method inherits the inherent robustness to observational noise characteristic of neural networks and effectively circumvents the curse of dimensionality, making it well-suited for high-dimensional parameter spaces. In contrast to purely deep learning-based methods, our approach leverages an alternating iterative structure to guide the optimization process, thereby significantly accelerating convergence and enhancing computational efficiency, particularly during the early stages of training. In summary, the primary contributions of this work are three-fold:
\begin{itemize}
\item We introduce the Alternating Learning-Based Collocation (ALBC) method, a novel framework tailored for solving inverse elliptic problems.
\item We establish rigorous stability estimates and provide a comprehensive convergence analysis for the proposed ALBC framework.
\item Through extensive numerical experiments, we demonstrate the effectiveness and superiority of ALBC over both traditional collocation methods and several deep-learning baselines, achieving higher accuracy with lower computational cost.
\end{itemize}
The remainder of this paper is organized as follows. Section \ref{sec1_prelim} outlines the problem formulation and the collocation method employed for its solution. Section \ref{sec1_method} details the proposed framework. Section \ref{sec_analysis} establishes the theoretical guarantees, including stability and convergence analysis. In Section \ref{sec1_result}, we present numerical experiments that assess the performance of the proposed approach across a range of inverse elliptic problems. Finally, Section \ref{sec1_conclusion} summarizes our findings and discusses possible directions for future research.

\section{Preliminary}\label{sec1_prelim}

In this section, we define the target problems and the foundational optimization framework discussed in this paper. Section \ref{subsec:problem} outlines the formal problem setting, while Section \ref{subsec:collocation} reviews the standard collocation scheme for inverse problems, which serves as the basis for our proposed method.

\subsection{Problem setting}\label{subsec:problem}

Let $\Omega \subset \mathbb{R}^d$ be a bounded domain with a Lipschitz continuous boundary $\partial\Omega$. We consider a physical system governed by the following second-order elliptic boundary value problem:
\begin{equation}
\label{eq:elliptic_system}
\left\{
\begin{aligned}
    -\nabla \cdot (q(x) \nabla u(x)) + b(x)u(x) &= f(x), & \quad x &\in \Omega, \\
    \mathcal{B}u(x) &= g(x), & \quad x &\in \partial\Omega.
\end{aligned}
\right.
\end{equation}
Here, $q(x)$ represents the diffusion coefficient (or conductivity), assumed to be strictly positive to satisfy the ellipticity condition. The term $b(x) \ge 0$ denotes the reaction coefficient (or potential field), and $f(x)$ is the internal source term. The operator $\mathcal{B}$ specifies the boundary condition including Dirichlet, Neumann, with $g(x)$ representing the boundary data. In this work, the objective is to identify an unknown parameter field, generally denoted by $\lambda$, while assuming all other system parameters are given. Depending on the physical scenario, the target unknown $\lambda$ in this study may represent one of the following:
\begin{itemize}
    \item The diffusion coefficient $q(x)$ (e.g., conductivity imaging);
    \item The reaction coefficient $b(x)$ (e.g., potential reconstruction);
    \item The source term $f(x)$ (e.g., source identification problems).
\end{itemize}
To reconstruct the target field $\lambda$, we utilize observational data $\mathcal{L}_{obs}$ governed by the measurement equation:
\begin{equation}\label{eq:observation}
    \mathcal{L}_{obs}(\mathbf{x}) = \mathcal{O}_{l}(u)(\mathbf{x}) + \epsilon(\mathbf{x}), \quad \mathbf{x} \in \Omega,
\end{equation}
where $\mathcal{O}_l$ denotes a linear observation operator and $\epsilon$ is an additive noise term of level $\delta_{noise} \geq 0$. Notably, the method presented in this work can be readily extended to accommodate observation models that exhibit a bilinear or even nonlinear dependence on both the state $u$ and the parameter $\lambda$, as demonstrated by the Current Density Impedance Imaging (CDII) experiment in Section \ref{sec1_result}.

To facilitate the subsequent derivation, we define the residual operator as $$\mathcal{N}(u, \lambda)(x) = -\nabla \cdot (q(x) \nabla u(x)) + b(x)u(x) - f(x),$$where, depending on the specific inverse problem, exactly one of the functions $q$, $b$, or $f$ is identified with the unknown $\lambda$, while the remaining two are treated as known data.
Although $\mathcal{N}$ is nonlinear with respect to the pair $(u, \lambda)$, it possesses a favorable bi-linear structure. Specifically, we denote $\mathcal{N}_u(u):=-\nabla \cdot (q(x) \nabla u(x)) + b(x)u(x)$ as the linear operator acting on $u$ when $\lambda$ is fixed, and $\mathcal{N}_\lambda(\lambda)$ as the linear operator acting on $\lambda$ when $u$ is given (e.g., $\mathcal{N}_\lambda(\lambda) := \lambda(x)u(x)$ for potential identification). Crucially, both $\mathcal{N}_u$ and $\mathcal{N}_\lambda$ exhibit strict linearity with respect to their primary arguments, $u$ and $\lambda$, respectively.

\subsection{Collocation scheme for inverse problem}\label{subsec:collocation}
The collocation method provides a rigorous framework for discretizing infinite-dimensional operator equations by enforcing governing laws strongly at a set of nodal points. In contrast to weak-formulations that rely on integral projections, the collocation scheme directly minimizes the pointwise residuals of the differential operators. For the inverse problem defined in \eqref{eq:elliptic_system}, this method approximates the unknown state $u(x)$ and parameter $\lambda(x)$ within a finite-dimensional subspace $\mathcal{V}_{N} = \text{span}\{\phi_j\}_{j=1}^{N}$. Specifically, one can adopt a unified ansatz where both fields are expanded using the same set of basis functions $\{\phi_j\}_{j=1}^{N}$:
\begin{equation}
    u(x) \approx u_N(x) = \boldsymbol{\phi}(x)^\top \mathbf{c}, \qquad \lambda(x) \approx \lambda_N(x) = \boldsymbol{\phi}(x)^\top \mathbf{d},
\end{equation}
where $\boldsymbol{\phi}(x) = [\phi_1(x), \dots, \phi_N(x)]^\top$ is the basis vector, $\mathbf{c}, \mathbf{d} \in \mathbb{R}^{N}$ are the coefficient vectors to be determined. 

To formulate the discrete system, we define a total of $M$ collocation points, partitioned into interior points $\mathcal{X}_{int} = \{x_i\}_{i=1}^{M_{int}} \subset \Omega$, boundary points $\mathcal{X}_{bdy} = \{x_b\}_{b=1}^{M_{bdy}} \subset \partial\Omega$, and observation points $\mathcal{X}_{obs}=\{x_o\}_{o=1}^{M_{obs}}\subset \Omega$, such that $M = M_{int} + M_{bdy} + M_{obs}$. Evaluating the state $u$ and parameter $\lambda$ at these collocation points and substituting them into the governing equations \eqref{eq:elliptic_system} yields the following interior physical residual vector $\mathbf{r}_{\Omega} \in \mathbb{R}^{M_{int}}$ and the boundary residual vector $\mathbf{r}_{\mathcal{B}} \in \mathbb{R}^{M_{bdy}}$:

$$[\mathbf{r}_{\Omega}(\mathbf{c},\mathbf{d})]_i := \mathcal{N}\left( \boldsymbol{\phi}^\top \mathbf{c} \, , \, \boldsymbol{\phi}^\top \mathbf{d} \right)(x_i), \quad x_i \in \mathcal{X}_{int},$$

$$[\mathbf{r}_{\mathcal{B}}(\mathbf{c})]_b := \mathcal{B}\left( \boldsymbol{\phi}^\top \mathbf{c} \right)(x_b) - g(x_b), \quad x_b \in \mathcal{X}_{bdy}.$$

Given the bilinearity of $\mathcal{N}$, the coupled algebraic system $\mathbf{r}_{\Omega}$ renders the inverse problem inherently nonlinear. A common approach to addressing this is to simultaneously optimize the independent variables $(\mathbf{c}, \mathbf{d}) \in \mathbb{R}^{2N}$ by relaxing the governing equations into soft penalty terms \cite{leeuwen2016penalty}:
\begin{equation}\label{eq:joint_opt_cd}
    (\mathbf{c}^*, \mathbf{d}^*) = \operatorname*{argmin}_{\mathbf{c}, \mathbf{d}} \left\{ \frac{w_r}{M_{int}} \| \mathbf{r}_{\Omega}(\mathbf{c}, \mathbf{d}) \|_2^2 + \frac{w_b}{M_{bdy}} \| \mathbf{r}_{\mathcal{B}}(\mathbf{c}) \|_2^2 + \frac{w_d}{M_{obs}} \| \mathcal{O}_l(\boldsymbol{\phi}^\top \mathbf{c}) - \mathcal{L}_{obs} \|_2^2 + \gamma \mathcal{R}_{reg}(\mathbf{d}) \right\}.
\end{equation}
Here, $w_r$, $w_b$, and $w_d$ are positive penalty weights balancing the physics residual, boundary condition, and data fidelity terms, respectively. $\|\cdot\|_2$ denotes the discrete $l_2$ norm over the corresponding set of collection points. The term $\mathcal{R}_{reg}(\mathbf{d})$ denotes a regularization functional that promotes smoothness or sparsity of the reconstructed parameter.

Alternatively, by enforcing strict physical constraints $\mathbf{r}_{\Omega}(\mathbf{c}, \mathbf{d}) = \mathbf{0}$ and $\mathbf{r}_{\mathcal{B}}(\mathbf{c}) = \mathbf{0}$, the state coefficients are reduced to an implicit function of the parameters, establishing the mapping $\mathbf{d} \mapsto \mathbf{c}(\mathbf{d})$. This condenses the optimization to a subspace $\mathbb{R}^N$, yielding the reduced objective function \cite{li2015local,norgren2005chebyshev}:

$$\mathbf{d}^* = \operatorname*{argmin}_{\mathbf{d}} \left\{ \frac{w_d}{M_{obs}} \| \mathcal{O}_l(\boldsymbol{\phi}^\top \mathbf{c}(\mathbf{d})) - \mathcal{L}_{obs} \|_2^2 + \gamma \mathcal{R}_{reg}(\mathbf{d}) \right\}.$$
While the former circumvents iterative forward PDE solutions at the cost of exacerbated non-convexity, the latter guarantees exact physical fidelity but necessitates full numerical inversions and adjoint-based gradient evaluations at each iteration.

In practice, the implementation of this framework faces a critical bottleneck regarding the determination of an optimal basis set. As previously noted, selecting an appropriate basis is a non-trivial task. The formulation of the inverse problem provides limited prior guidance, primarily because the regularity of the unknown parameter field is not pre-defined. Consequently, an inappropriate choice of basis inevitably leads to severe numerical instabilities:

\begin{enumerate}
    \item \textbf{Spectral mismatch and Gibbs phenomenon.} Using globally smooth basis functions to approximate a parameter field with local abrupt changes or discontinuities results in spurious high-frequency oscillations known as the Gibbs phenomenon \cite{gottlieb1997gibbs}, which significantly degrades reconstruction accuracy.
    \item \textbf{Ill-conditioning}: An improper basis can lead to a rapidly growing condition number of the resulting linear system (or differentiation matrices) \cite{kansa2017ill}, causing numerical instability where small errors in data or floating-point arithmetic are catastrophically amplified.
\end{enumerate}

To overcome these limitations, we propose a learning-based collocation method. By leveraging neural networks to construct adaptive basis functions, this framework effectively mitigates the aforementioned numerical pathologies and demonstrates high fidelity in solving complex inverse problems.

\section{Methodology}\label{sec1_method}

In this section, we present the \textbf{A}lternating \textbf{L}earning-\textbf{B}ased \textbf{C}ollocation (ALBC) method designed for inverse elliptic problems. While traditional collocation schemes typically rely on static, pre-defined basis functions, our approach leverages shallow neural networks to construct basis sets dynamically. This flexibility enables the alternating generation of bases for coupled state and parameter fields, offering two distinct advantages. First, it allows for the adaptive and incremental expansion of the basis dimensionality. Second, it ensures that the basis functions inherently capture the structural characteristics of the underlying physical problem. As a result, ALBC exhibits superior generalizability and robustness.

The remainder of this section is organized as follows. Section \ref{sec_main} provides a high-level overview of the proposed framework and presents the core algorithm. Subsequent sections detail the technical specifications. Specifically, Section \ref{sec_basis} describes the construction and training protocols for the basis functions. Section \ref{sec_ini} outlines the initialization strategy and Section \ref{sec_sampling} introduces the adaptive sampling method.

\begin{algorithm}[ht!]
\SetAlgoLined
\DontPrintSemicolon

\SetKwComment{Comment}{}{} 

\caption{Alternating Learning-Based Collocation (ALBC) with Periodic Fine-tuning}
\label{main_algo}

%
%
%
%
%
%
%

\KwIn{Target basis count $s$, network widths $\{n_k\}_{k=1}^s$, fine-tuning period $T$}
\KwOut{Final approximations $u_{s}$ and $\lambda_{s}$}

\BlankLine
Initialize empty basis sets: $\mathcal{U} \gets \emptyset$, $\mathcal{P} \gets \emptyset$\;
Initialize approximations: $u \gets 0$, $\lambda \gets 0$\;

\BlankLine
\For{$k \gets 1$ \KwTo $s$}{
    \BlankLine
    \Comment{\# Step 1: Update basis set $\mathcal{U}$ and state variable $u$}
    Initialize the $k$-th basis $\phi_k$ of width $n_k$ as described in Section~\ref{sec_ini}\;
    Sample collocation points according to Section~\ref{sec_sampling}\;
    Train the basis $\phi_k$ following Section~\ref{sec:update_u}\;
    Append to basis set: $\mathcal{U} \gets \mathcal{U} \cup \{\phi_k\}$\;
    Update state variable: $u \gets u + \phi_k$\;
    
    \BlankLine
    \Comment{\# Step 2: Update basis set $\mathcal{P}$ and parameter field $\lambda$}
    Initialize the $k$-th basis $\psi_k$ of width $n_k$ as described in Section~\ref{sec_ini}\;
    Sample collocation points according to Section~\ref{sec_sampling}\;
    Train the basis $\psi_k$ following Section~\ref{sec:update_lambda}\;
    Append to basis set: $\mathcal{P} \gets \mathcal{P} \cup \{\psi_k\}$\;
    Update parameter field: $\lambda \gets \lambda + \psi_k$\;

    \BlankLine
    \Comment{\# Step 3: Update global coefficients $(\mathbf{c},\mathbf{d})$}
    Update $(\mathbf{c},\mathbf{d})$ by solving the joint optimization problem \eqref{eq:joint_opt_cd}\;

    \BlankLine
    \Comment{\# Step 4: Periodic joint fine-tuning}
    \If{$k \pmod{T} = 0$}{
        Jointly fine-tune $(\mathbf{c},\mathbf{d})$ and all parameters in $\mathcal{U}$ and $\mathcal{P}$ per Section~\ref{sec_basis}\;
    }
}

\BlankLine
\Return{$u_{s}$, $\lambda_{s}$}

\end{algorithm}

\subsection{The Alternating Learning-based Collocation method}\label{sec_main}
Unlike the standard collocation framework in Subsection \ref{subsec:collocation}, which uses a single shared basis for both $u$ and $\lambda$, we employ two distinct basis sets $\mathcal{U}$ and $\mathcal{P}$, parameterized by separate shallow neural networks. This separation is motivated by the fact that $u$ and $\lambda$ may exhibit fundamentally different regularity and spectral characteristics, and constructing dedicated bases allows each to be tailored to its respective target function. Given that the observational data is restricted to $u$, we implement an alternating generation strategy starting with the state variable. As delineated in Algorithm \ref{main_algo}, the iterative procedure commences with empty basis sets and sequentially expands them up to a target dimension $s$.

At the $k$-th iteration, a new neural network basis function $\phi_k$ with width $n_k$ is initialized and trained to minimize the data-driven loss; this function is subsequently used to enrich the state basis set $\mathcal{U}$. Following this, a corresponding basis function $\psi_k$ for the parameter field $\lambda$ is introduced and optimized by penalizing the PDE residual in conjunction with necessary regularization, before being incorporated into $\mathcal{P}$. Based on these augmented basis sets, the current approximations $u_k$ and $\lambda_k$ are updated via the collocation method. To mitigate error propagation and maintain global physical consistency during the sequential expansion, a joint optimization is executed every $T$ stages to simultaneously re-calibrate all accumulated basis functions in $\mathcal{U}$ and $\mathcal{P}$. This alternating procedure continues until the basis dimensionality reaches a predefined threshold $s$, culminating in a final collocation update to yield the optimal approximations $u_s$ and $\lambda_s$.

\subsection{Basis construction and training procedure}\label{sec_basis}

Constructing effective basis functions requires a parameterization that is both easily trainable and highly expressive. The Single-Hidden-Layer Neural Network (SHLNN) satisfies these requirements, offering a theoretically guaranteed universal approximation at a minimal computational cost \cite{chen2022bridging,cybenko1989approximation}. In this work, we utilize SHLNNs with sinusoidal activations to emulate the principle of Fourier series expansion. This framework allows the network to adaptively learn various frequency components, making it superior to ReLU or Tanh for capturing high-frequency oscillations.

Formally, we parameterize the basis functions for the state $u$ and the parameter $\lambda$ using neural networks as follows:

\begin{subequations} \label{eq:basis_networks}
\begin{align}
    \phi_k(\mathbf{x}; \theta_u^{(k)}) &= \sum_{j=1}^{n_k} a_{k,j}^u \phi_{k,j} = \sum_{j=1}^{n_k} a_{k,j}^u \sin(\boldsymbol{\omega}_{k,j}^u \cdot \mathbf{x} + b_{k,j}^u), \label{eq:basis_u} \\
    \psi_k(\mathbf{x}; \theta_\lambda^{(k)}) &= \sum_{j=1}^{n_k} a_{k,j}^{\lambda} \psi_{k,j} = \sum_{j=1}^{n_k} a_{k,j}^\lambda \sin(\boldsymbol{\omega}_{k,j}^\lambda \cdot \mathbf{x} + b_{k,j}^\lambda). \label{eq:basis_lambda}
\end{align}
\end{subequations}
where $n_k$ denotes the width of the $k$-th basis network, $\boldsymbol{\omega}_{k,j}^u, \boldsymbol{\omega}_{k,j}^\lambda \in \mathbb{R}^d$ represent the frequency vectors of the $j$-th neuron in the respective basis networks, and $\theta_u^{(k)} = \{a_{k,j}^u, \boldsymbol{\omega}_{k,j}^u, b_{k,j}^u\}_{j=1}^{n_k}$ (resp.\ $\theta_\lambda^{(k)}$) encompasses the trainable parameters. To express this in a more compact form, we define the coefficient column vector $\mathbf{a}^{u}_k := [a_{k,1}^{u}, \dots, a_{k,n_k}^{u}]^T$ and the corresponding row vector of neuron outputs:$$\boldsymbol{\phi}^{u}_k(\mathbf{x}) := \big[ \sin(\boldsymbol{\omega}_{k,1}^{u} \cdot \mathbf{x} + b_{k,1}^{u}), \dots, \sin(\boldsymbol{\omega}_{k,n_k}^{u} \cdot \mathbf{x} + b_{k,n_k}^{u}) \big].$$By applying analogous definitions for $\mathbf{a}^{\lambda}_k$ and $\boldsymbol{\psi}^{\lambda}_k(\mathbf{x})$, the parameterizations can be succinctly rewritten as $\phi_k(\mathbf{x}; \theta_u^{(k)}) = \boldsymbol{\phi}^{u}_k \mathbf{a}^{u}_k$ and $\psi_k(\mathbf{x}; \theta_\lambda^{(k)}) = \boldsymbol{\psi}^{\lambda}_k \mathbf{a}^{\lambda}_k$.

We now detail the optimization strategy. As outlined in Algorithm \ref{main_algo}, the training procedure is structured into two primary phases: (i) the sequential optimization of new basis functions to fit the current residuals (Lines 6 and 11), and (ii) update of $(\mathbf{c},\mathbf{d})$ by collection method (Line 14) and a joint fine-tuning phase where all accumulated model parameters are optimized simultaneously (Line 16). This strategy enforces a tight coupling between the variables, reinforcing their structural interdependencies and maximizing reconstruction accuracy.

\subsubsection{Update of \texorpdfstring{$\mathcal{U}$}{U} and \texorpdfstring{$u$}{u}}\label{sec:update_u}

Let $u_{k-1}$ and $\lambda_{k-1}$ denote the current global approximations. The localized target residuals during the $k$-th stage are defined across their respective domains as:

$$r^{(k)}_{data}(\mathbf{x}) := \mathcal{L}_{obs}(\mathbf{x}) - \mathcal{O}_{l}(u_{k-1})(\mathbf{x}),$$

$$r^{(k)}_{bdy}(\mathbf{x}) := g(\mathbf{x}) - \mathcal{B}(u_{k-1})(\mathbf{x}),$$

$$r^{(k)}_{phy}(\mathbf{x}) := -\mathcal{N}(u_{k-1},\lambda_{k-1})(\mathbf{x}).$$

The pointwise evaluations of these residual functions at their respective collocation sets are assembled into residual vectors:
$$\mathbf{r}^{(k)}_{data} := \bigl[r^{(k)}_{data}(\mathbf{x}_1),\dots,r^{(k)}_{data}(\mathbf{x}_{M_{obs}})\bigr]^\top \in \mathbb{R}^{M_{obs}},$$
and analogously $\mathbf{r}^{(k)}_{bdy} \in \mathbb{R}^{M_{bdy}}$ and $\mathbf{r}^{(k)}_{phy} \in \mathbb{R}^{M_{int}}$.

With the previous approximations fixed, we construct the $k$-th basis function $\phi_k(\mathbf{x}; \theta_u^{(k)})$ by minimizing the current approximation residuals. Exploiting the linearity of the operators ($\mathcal{O}_{l}$, $\mathcal{B}$, and $\mathcal{N}_u$) with respect to $u$, we formulate a composite loss function that penalizes discrepancies in the observational data, boundary conditions, and physical governing equations:
\begin{equation}\label{u_unnorm}
\begin{aligned}
    \mathcal{R}^{(k)}_u(\theta_u^{(k)}) 
    &= \frac{1}{M_{obs}}\sum_{\mathbf{x} \in \mathcal{X}_{obs}} \left| \mathcal{O}_{l}(\phi_k)(\mathbf{x}) - r^{(k)}_{data}(\mathbf{x}) \right|^2 \\
    &\quad + \frac{\eta_1}{M_{bdy}} \sum_{\mathbf{x} \in \mathcal{X}_{bdy}} \left| \mathcal{B}(\phi_k)(\mathbf{x}) - r^{(k)}_{bdy}(\mathbf{x}) \right|^2 \\
    &\quad + \frac{\eta_2}{M_{int}} \sum_{\mathbf{x} \in \mathcal{X}_{int}} \left| \mathcal{N}_u(\phi_k)(\mathbf{x}) - r^{(k)}_{phy}(\mathbf{x}) \right|^2 \\
    &\coloneqq \mathcal{R}_{data}^{(k)} + \eta_1 \mathcal{R}_{bdy}^{(k)} + \eta_2 \mathcal{R}_{phy}^{(k)}.
\end{aligned}
\end{equation}
Here, $\eta_1 > 0$ and $\eta_2 > 0$ are positive penalty weights that balance the boundary and physics residuals relative to the data term, respectively.

\subsubsection{Update of \texorpdfstring{$\mathcal{P}$}{P} and \texorpdfstring{$\lambda$}{lambda}}\label{sec:update_lambda}
Once $u_k$ is obtained, we update the parameter field $\lambda$ and the corresponding basis set $\mathcal{P}$. Exploiting the linearity of the differential operator $\mathcal{N}_{\lambda}$ with respect to $\lambda$, we construct the $k$-th basis function $\psi_k(\mathbf{x}; \theta_\lambda^{(k)})$ by minimizing the  physical residuals:
\begin{equation}\label{lambda_unnorm}
\mathcal{R}^{(k)}_\lambda(\theta_\lambda^{(k)}) = \frac{1}{M_{int}}\sum_{\mathbf{x} \in \mathcal{X}_{int}} \left| \mathcal{N}_\lambda(\psi_k)(\mathbf{x}) - r^{(k)}_{\lambda}(\mathbf{x}) \right|^2 + \gamma_\lambda \sum_{j=1}^{n_k} \Vert\boldsymbol{\omega}_{k,j}^{\lambda}\Vert_2 := \mathcal{R}_{phy,\lambda}^{(k)} + \gamma_\lambda \mathcal{R}_{reg,\lambda}^{(k)},
\end{equation}
where $r^{(k)}_{\lambda}$ is the intermediate physical residual defined as:
\begin{equation}
    r^{(k)}_{\lambda}(\mathbf{x}) \mathrel{\mathop:}= -\mathcal{N}(u_{k}, \lambda_{k-1})(\mathbf{x}),
\end{equation}
and $\gamma_\lambda$ is the penalty coefficient. This group Lasso regularization acts as a low-pass filter, suppressing spurious oscillations and biasing the network toward smoother, physically plausible structures. 

After refining the basis function sets $\mathcal{U}$ and $\mathcal{P}$, we hold the basis parameters constant. The global coefficients $(\mathbf{c},\mathbf{d})$ are then computed using the standard collocation method for \eqref{eq:joint_opt_cd}. For the inverse source problem, this yields a linear system. However, the potential and diffusion coefficient identification problems necessitate solving a nonlinear optimization problem. To circumvent this, we introduce the scheme detailed in Algorithm \ref{main_algo_nonlinear}. We decouple the original nonlinear problem \eqref{eq:joint_opt_cd} into two least-squares problems. By fixing one set of coefficients as constants, each subproblem can be efficiently solved as a linear system.

\begin{algorithm}[ht!]
\SetAlgoLined
\DontPrintSemicolon

\SetKwComment{Comment}{}{} 

\caption{ALBC for inverse potential/diffussion coefficient problem}
\label{main_algo_nonlinear}

\KwIn{Target basis count $s$, network widths $\{n_k\}_{k=1}^s$, fine-tuning period $T$}
\KwOut{Final approximations $u_{s}$ and $\lambda_{s}$}

\BlankLine
Initialize empty basis sets: $\mathcal{U} \gets \emptyset$, $\mathcal{P} \gets \emptyset$\;
Initialize approximations: $u \gets 0$, $\lambda \gets 0$\;

\BlankLine
\For{$k \gets 1$ \KwTo $s$}{
    \BlankLine
    \Comment{\# Step 1: Update basis set $\mathcal{U}$ and state variable $u$}
    Initialize the $k$-th basis $\phi_k$ of width $n_k$ as described in Section~\ref{sec_ini}\;
    Sample collocation points according to Section~\ref{sec_sampling}\;
    Train the basis $\phi_k$ following Section~\ref{sec:update_u}\;
    Append to basis set: $\mathcal{U} \gets \mathcal{U} \cup \{\phi_k\}$\;
    Update state variable: $u \gets u + \phi_k$\;
    Update coefficient $\mathbf{c}$ by solving \eqref{eq:joint_opt_cd} as a least square problem with $\mathbf{d} =\mathbf{d}_{k-1}$\;
    
    \BlankLine
    \Comment{\# Step 2: Update basis set $\mathcal{P}$ and parameter field $\lambda$}
    Initialize the $k$-th basis $\psi_k$ of width $n_k$ as described in Section~\ref{sec_ini}\;
    Sample collocation points according to Section~\ref{sec_sampling}\;
    Train the basis $\psi_k$ following Section~\ref{sec:update_lambda}\;
    Append to basis set: $\mathcal{P} \gets \mathcal{P} \cup \{\psi_k\}$\;
    Update parameter field: $\lambda \gets \lambda + \psi_k$\;
    Update coefficient $\mathbf{d}$ by solving \eqref{eq:joint_opt_cd} as a least square problem with $\mathbf{c} =\mathbf{c}_{k}$\;

    \BlankLine
    \Comment{\# Step 3: Periodic joint fine-tuning}
    \If{$k \pmod{T} = 0$}{
        Jointly fine-tune $(\mathbf{c},\mathbf{d})$ and all parameters in $\mathcal{U}$ and $\mathcal{P}$ per Section~\ref{sec_basis}\;
    }
}

\BlankLine
\Return{$u_{s}$, $\lambda_{s}$}

\end{algorithm}

\subsubsection{Periodic joint fine-tuning}
The alternating training procedure introduced in the previous subsections is fundamentally a greedy strategy \cite{weng2026deep,xu2025randomized}. Under this framework, previously established basis functions remain frozen; the updated approximation is obtained merely by superimposing the newly generated basis. While computationally efficient, this sequential construction suffers from myopic optimization. To overcome this, we update all trainable parameters $\Theta = \{c_i, d_i, \theta_u^{(i)}, \theta_\lambda^{(i)}\}_{i=1}^{k}$ periodically by minimizing the comprehensive loss defined as:

\begin{equation}
\begin{split}
\mathcal{R}_{total}(\Theta) &= \frac{w_r}{M_{int}} \sum_{\mathbf{x} \in \mathcal{X}_{int}} \left| \mathcal{N}(u_k, \lambda_k)(\mathbf{x}) \right|^2 \\
&\quad + \frac{w_b}{M_{bdy}} \sum_{\mathbf{x} \in \mathcal{X}_{bdy}} \left| \mathcal{B}(u_k)(\mathbf{x}) - g(\mathbf{x}) \right|^2 \\
&\quad + \frac{w_d}{M_{obs}} \sum_{\mathbf{x} \in \mathcal{X}_{obs}} \left| \mathcal{O}_l(u_k)(\mathbf{x}) - \mathcal{L}_{obs}(\mathbf{x}) \right|^2 \\
&\quad + \mathcal{R}_{freq}(\Theta),
\end{split}
\end{equation}
where $w_r, w_b, w_d \ge 0$ are penalty weights, and the frequency regularization term $\mathcal{R}_{freq}(\Theta)$ enforces smoothness and prevents overfitting for both networks across all actively accumulated basis functions:

$$\mathcal{R}_{freq}(\Theta) = \gamma_u \sum_{i=1}^{k} \sum_{j=1}^{n_i} \Vert \boldsymbol{\omega}_{i,j}^{u} \Vert_2 + \gamma_\lambda \sum_{i=1}^{k} \sum_{j=1}^{n_i} \Vert \boldsymbol{\omega}_{i,j}^{\lambda} \Vert_2,$$
where $\gamma_u > 0$ is the frequency penalty weight for the state basis, and $\gamma_\lambda > 0$ is the corresponding weight for the parameter basis (the same $\gamma_\lambda$ appearing in \eqref{lambda_unnorm}).

Fundamentally, this joint fine-tuning step is mathematically equivalent to training a standard PINN equipped with the SHLNN architecture. However, as demonstrated in our numerical examples, preceding this step with the alternating basis generation significantly accelerates convergence compared to standard PINN training (see Figure \ref{source:compare}). Furthermore, integrating this joint optimization phase into the alternating framework is crucial for maximizing the overall approximation accuracy (see Figure \ref{fine-tuning}). This synergistic effect is systematically analyzed in Section \ref{sec1_result}.

\subsection{Initialization}\label{sec_ini}
Given that the frequency characteristics of the neural network are governed exclusively by the input layer parameters $\boldsymbol{\omega}$, a well-calibrated initialization strategy is pivotal for facilitating rapid convergence and enhancing training efficacy.

To this end, we define a hyperparameter boundary $R_k$, where the $k$-th hidden layer parameters of the base network are initialized by sampling uniformly from the interval $[-R_k, R_k]$. As demonstrated in \cite{rahaman2019spectral,xu2019frequency}, neural networks exhibit a spectral bias, showing a preference for learning low-frequency components during the early stages of training; consequently, an excessively large $R_k$ produces oversized hidden parameters, causing the network to introduce high frequency noise and squander its representational capacity, whereas an overly small $R_k$ restricts the spectral range and prevents the capture of crucial high frequency details. In fact, the choice of $R_k$ reflects the intrinsic expressive power of the neural network architecture itself~\cite{chen2022bridging}.


\begin{remark}
    While \cite{muller1959note,xu2025randomized} advocates for spherical sampling to mitigate the anisotropic ``corner effects'' inherent in hypercube sampling, these effects are minimal in the low-dimensional problems considered here. Given that the impact on the frequency distribution is marginal in this context, we opt for the computationally simpler strategy of direct uniform sampling on the interval $[-R_k, R_k]$.
\end{remark}

\subsubsection{Frequency-Domain Analysis}

For the subsequent update, we employ a tailored frequency-domain analysis to determine value of $R_k$. This is a two-step procedure: first, we define an analysis field based on the specific observation format; second, we apply the Discrete Fourier Transform (DFT) to this field to identify its dominant frequencies.

\begin{enumerate}
  \item \textbf{Determine the analysis field.} For the update of $u$, the data observation term $\mathcal{R}_{\text{data}}^{(k)}$ is assigned the dominant weight in \eqref{u_unnorm} (i.e., $\eta_1, \eta_2 < 1$). Therefore, the frequency initialization is naturally guided by the data residual $r^{(k)}_{\text{data}}$. For the parameter field $\lambda$, the loss $\mathcal{R}^{(k)}_\lambda$ in \eqref{lambda_unnorm} is purely equation-based. Since the basis $\psi_k$ is trained to satisfy $\mathcal{N}_\lambda(\psi_k) \approx r^{(k)}_\lambda$, a natural initialization strategy is to match the dominant frequencies of the approximated solution $\mathcal{N}_\lambda^{-1} r^{(k)}_\lambda$. Specifically, for the source identification problem ($\lambda = f$, $\mathcal{N}_\lambda(\psi) = -\psi$), direct inversion yields $\psi_k \approx -r^{(k)}_\lambda$, making $r^{(k)}_\lambda(\mathbf{x})$ the appropriate analysis field. Conversely, for the potential reconstruction problem ($\lambda = b$, $\mathcal{N}_\lambda(\psi) = u_k\psi$), pointwise inversion of $u_k\psi_k = r^{(k)}_\lambda$ yields the analysis field $r^{(k)}_\lambda(\mathbf{x})\,/\,u_k(\mathbf{x})$.
  
  \item \textbf{Calculate the dominant frequencies.} Given the analysis field obtained above, we compute its discrete Fourier transform over the uniform grid $\mathcal{X}_{\text{FFT}}$ defined by \eqref{eq:unif_grid}. We then select the top $n_{\text{fft}}$ wavevectors associated with the largest amplitudes and assign them as the initial input weights $\boldsymbol{\omega}_{k,j}$ in \eqref{eq:basis_networks}. The remaining $n_k - n_{\text{fft}}$ frequencies are sampled uniformly from $[-R_k, R_k]$, where $$R_k := \min\left\{\omega_{\max}, \frac{1}{4}\sum_{i=1}^{4} \omega_i^{\text{fft}}\right\}.$$ Here, $\omega_{\max}$ acts as a safeguard against high-frequency instability, while the average of the top four wavevectors $\{\omega_i^{\text{fft}}\}_{i=1}^4$ dynamically scales the sampling bound to capture the intrinsic physical features of the current solution.
\end{enumerate}

\subsection{Collocation point strategy}\label{sec_sampling}

The method relies on two distinct but complementary sets of points, each serving a dedicated purpose.

\paragraph{FFT grid $\mathcal{X}_{\mathrm{FFT}}$.}
As described in Section~\ref{sec_ini}, the frequency-domain initialization requires evaluating the residual on a structured uniform grid. For a rectangular domain $\Omega = \prod_{k=1}^{d}[a_k, b_k]$, this grid is defined as
\begin{equation}\label{eq:unif_grid}
    \mathcal{X}_{\mathrm{FFT}} = \prod_{k=1}^{d} \left\{ a_k + \frac{i_k}{m_k}(b_k - a_k) : i_k = 0, 1, \ldots, m_k \right\},
\end{equation}
yielding $N_{\mathrm{FFT}} = \prod_{k=1}^{d}(m_k+1)$ points in total, where $m_k$ denotes the number of grid intervals along the $k$-th coordinate. For domains with irregular boundaries, $\mathcal{X}_{\mathrm{FFT}}$ is constructed via rejection sampling \cite{von1963various}. This grid is used exclusively for frequency analysis and does not serve as training collocation points.

\paragraph{Training collocation set $\mathcal{X}_{\mathrm{int}}$.}
To enhance training efficiency for multi-scale or non-homogeneous solutions, we adopt a residual-based hybrid collocation strategy. At the beginning of each stage, the interior points $\mathcal{X}_{\mathrm{int}}$ are drawn from a mixture distribution \cite{jiao2024gaussian}:
\begin{equation}\label{eq:rad}
    \rho(\mathbf{x}) = \alpha\,\rho_{\mathrm{unif}}(\mathbf{x}) + \beta\,\rho_{\mathrm{adapt}}(\mathbf{x}), \quad \alpha + \beta = 1,
\end{equation}
where $\rho_{\mathrm{unif}}$ is the uniform distribution over $\Omega$, and $\rho_{\mathrm{adapt}}$ is a residual-driven importance distribution. Specifically, letting $M_{\mathrm{unif}} = \alpha M_{\mathrm{int}}$ and $M_{\mathrm{adapt}} = \beta M_{\mathrm{int}}$, we draw $M_{\mathrm{unif}}$ points from $\rho_{\mathrm{unif}}$ and select the remaining $M_{\mathrm{adapt}}$ points with probabilities proportional to the squared data residual $r^{(k)}_{\mathrm{data}}$~\cite{wu2023comprehensive}:
\begin{equation}\label{eq:prob_u}
    p(\mathbf{x}_i) = \frac{\bigl|r^{(k)}_{\mathrm{data}}(\mathbf{x}_i)\bigr|^2 + \delta_0}{\displaystyle\sum_{j=1}^{M_{\mathrm{obs}}} \bigl(\bigl|r^{(k)}_{\mathrm{data}}(\mathbf{x}_j)\bigr|^2 + \delta_0\bigr)}.
\end{equation}
Here $\delta_0 > 0$ is a small constant introduced for numerical stability. This choice is rooted in the principle of importance sampling: sampling proportionally to $|r^{(k)}_{\mathrm{data}}|^2$ minimizes the variance of the Monte Carlo estimator for the data-fidelity loss. As the approximated solution $u_{k-1}$ evolves, this distribution is dynamically updated at each stage. Finally, the boundary collocation set $\mathcal{X}_{\mathrm{bdy}}$ is drawn uniformly from $\partial \Omega$. 

To address scenarios where the available observation samples are insufficient, we introduce an observation smoothing network \cite{ulyanov2018deep}. This network is trained on the actual noisy measurements to approximate a continuous observation function. The learned surrogate function can subsequently be used to generate pseudo-observation data in arbitrarily large quantities, thereby facilitating a significantly more robust reconstruction.

\section{Numerical analysis}\label{sec_analysis}

Building upon extensive research \cite{el2016lipschitz,el2011inverse,ji2025potential} concerning the stability and uniqueness of inverse problems for elliptic equations, we assume that the inverse problem under consideration satisfies stability and uniqueness conditions. Let $(u^\dagger, \lambda^\dagger)$ denote the exact solution pair. We define the continuous population loss as

\begin{equation}\label{risk}
L(u, \lambda) = \mathcal{R}_{obs}(u) + \mathcal{R}_{phy}(u, \lambda) + \mathcal{R}_{bdy}(u),
\end{equation}
where
\[
\mathcal{R}_{obs}(u) = \| \mathcal{O}_l(u) - \mathcal{L}_{obs} \|_{L^2(\Omega)}^2, \quad \mathcal{R}_{phy}(u, \lambda) = \| \mathcal{N}(u, \lambda) \|_{L^2(\Omega)}^2, \quad \mathcal{R}_{bdy}(u) = \| \mathcal{B} u - g \|_{L^2(\partial \Omega)}^2.
\]
By design, $\mathcal{R}_{phy}(u^\dagger, \lambda^\dagger) = 0$ and $\mathcal{R}_{bdy}(u^\dagger) = 0$. In particular, the total loss satisfies $L(u^\dagger, \lambda^\dagger) = 0$ in the noise-free setting, whereas $\mathcal{R}_{obs}(u^\dagger) = \mathcal{O}(\delta_{noise}^2)$ under noisy observations. In the following, $\bar{\mathcal{R}}_u^{(k)}$ and $\bar{\mathcal{R}}_\lambda^{(k)}$ denote the stage-$k$ losses~\eqref{u_unnorm} and~\eqref{lambda_unnorm} expressed in their $L^2$ form. We analyze the asymptotic behavior of both $\bar{\mathcal{R}}_u^{(k)}$ and $\bar{\mathcal{R}}_\lambda^{(k)}$, establishing that the approximate solution $(u_k, \lambda_k)$ converges to the true solution $(u^{\dagger}, \lambda^{\dagger})$ as the number of stages $k \to \infty$. Since the sine activation function is smooth, non-constant, and bounded, the universal approximation property (Lemma~\ref{UOT}) is guaranteed.

For simplicity, we restrict our numerical analysis to inverse source and inverse potential problems, leaving the identification of diffusion coefficients for future work. Detailed proofs of the theorems in this section are deferred to the Appendix B to E.

\begin{assumption}\label{ass1}
    The state variable $u$ and the parameter $\lambda$ exhibit rapid spectral decay, i.e., they belong to $H^p(\Omega)$ with $p > \max \{d/2 + 1, 4\}$.
\end{assumption}

\begin{assumption}\label{ass2}
For potential inversion, we assume the state satisfies $|u^\dagger| \ge c_0 > 0$ a.e. in $\Omega$.
\end{assumption}

\begin{assumption}
\label{assump:uniform}
There exist finite constants $M_u, M_\lambda > 0$, depending only on $\omega_{\max}$, the initial loss $\bar{\mathcal{R}}_u^{(0)}$, and the priors $\|u^\dagger\|_{H^{p}(\Omega)}$, $\|\lambda^\dagger\|_{H^{p}(\Omega)}$, such that the iterates $\{u_k, \lambda_k\}_{k \ge 0}$ generated by the alternating algorithm satisfy
$$\|u_k\|_{W^{1,\infty}(\Omega)} \le M_u, \quad \|\lambda_k\|_{W^{1,\infty}(\Omega)} \le M_\lambda, \quad \forall k \ge 0.$$
\end{assumption}

\subsection{Stability}
\begin{theorem}[Stability of $u$]
\label{thm:u}
    Under Assumption~\ref{ass1} and for any $0 \le s < p$, the state error satisfies the deterministic bound:
    $$\Vert u_k - u^\dagger \Vert_{H^{s}(\Omega)}^2 \le C_{stab}^2 \, \omega_{\max}^{2s} \left( \bar{\mathcal{R}}^{(k)}_u + \delta_{noise}^2  \right)+ \mathcal{O}(\omega_{\max}^{-2(p-s)}),$$
    where $C_{stab} > 0$ is a stability constant depending only on the domain $\Omega$, the observation operator $\mathcal{O}_l$  and the PDE structure, independent of $\omega_{\max}$, $k$, and the network parameters.
\end{theorem}

\begin{theorem}[Stability of $\lambda$]\label{thm:lambda}
   Under Assumption~\ref{ass1}--\ref{assump:uniform}, the parameter reconstruction error satisfies the deterministic bound:
  $$\Vert \lambda_k - \lambda^\dagger \Vert_{L^2(\Omega)}^2 \le C_{\lambda}^2 \Big( \bar{\mathcal{R}}_{\lambda}^{(k)} + L_u \Vert u_k - u^\dagger \Vert_{H^2(\Omega)}^2 \Big),$$
  where $C_{\lambda} > 0$ is a stability constant depending only on the domain $\Omega$ and $L_u > 0$ is the Lipschitz coupling constant quantifying error propagation from the state variable.
\end{theorem}

\subsection{Convergence}

\begin{theorem}[Linear Convergence]
\label{thm:convergence}
Under Assumption~\ref{ass1}--\ref{assump:uniform}, suppose the basis networks are sufficiently wide, and $\eta_2$ is sufficiently small (see the details in proof). Then there exist weights $\kappa_1, \kappa_2 > 0$ such that the joint loss $\mathcal{J}^{(k)} := \kappa_1 \bar{\mathcal{R}}_u^{(k)} + \kappa_2 \bar{\mathcal{R}}_{\lambda}^{(k)}$ satisfies:
\begin{equation}
    \mathcal{J}^{(k)} \leq \rho^{k} \mathcal{J}^{(0)} + \frac{1-\rho^{k}}{1-\rho} \zeta,
\end{equation}
where $\rho \leq  7/8$ is a uniform contraction rate independent of $k$ and $\zeta = \mathcal{O}(\omega_{\max}^4\delta_{noise}^2 + \omega_{\max}^{-2(p-4)})$. In particular, as $k\to\infty$, $\bar{\mathcal{R}}_u^{\infty}=\mathcal{O}(\delta_{noise}^2+\omega_{\max}^{-2(p-2)})$ and $\bar{\mathcal{R}}_\lambda^{\infty}=\mathcal{O}(\omega_{\max}^4\delta_{noise}^2+\omega_{\max}^{-2(p-4)})$.
\end{theorem}

\begin{theorem}[Convergence rate]
\label{thm:balance}
By choosing $\omega_{\max}^\star\sim\delta_{noise}^{-1/(p-2)}$, as $k\to\infty$
we have
$$
\|u_{\infty}-u^\dagger\|_{L^2(\Omega)}^2=\mathcal{O}\!\big(\delta_{noise}^{2}\big),
\qquad
\|\lambda_{\infty}-\lambda^\dagger\|_{L^2(\Omega)}^2=\mathcal{O}\!\big(\delta_{noise}^{\,2(p-4)/(p-2)}\big).
$$
\end{theorem}

\begin{remark}
According to Theorem~\ref{thm:balance}, for a smooth target field, the $L^2$ errors of both $u$ and $\lambda$ converge at a rate of $\mathcal{O}(\delta_{noise})$. Although our theorem only provides an asymptotic convergence rate as $k\to\infty$, which might be conservative, this rate is readily observable within a finite number of stages in practice (see Figure~\ref{potential:noise_loglog} as an example).
\end{remark}

\section{Numerical Results}\label{sec1_result}
In this section, we evaluate the performance of our proposed method on several elliptic inverse problems. Table~\ref{tab:hyperparams} summarizes the hyperparameters for the network architecture, sampling complexity, and optimizer settings. For all examples, the number of neurons at the $k$-th stage follows a linear growth schedule: $n_k = 30 + 5(k-1)$. We use the Adam optimizer for the basis update and L-BFGS for the fine-tuning step. For noise levels of 10\% or higher, we train a 3-layer denoising network~\cite{ulyanov2018deep} with a tanh activation function on the noisy data. This trained network then serves as a surrogate observation model, unless otherwise specified.

\begin{table}[ht!]
\centering
\caption{Hyperparameters for each numerical example. $T$ is the joint fine-tuning period; $E_{Adam}^{u}$ and $E_{Adam}^{\lambda}$ are the Adam epochs in the every stage; $E_{LBFGS}$ is the L-BFGS fine-tuning iterations. $(E_{dn},\,\mathrm{lr}_{dn})$ are the epochs and learning rate of the denoising pre-processor; $\omega_{\max}$ is the frequency band-limit imposed on the basis networks.}
\label{tab:hyperparams}
\resizebox{\textwidth}{!}{%
\begin{tabular}{lccccccc}
\hline
Example & $T$ & $M_{obs}$/$M_{int}$ & Stages & $(E_{Adam}^{u},\,\mathrm{lr}) / (E_{Adam}^{\lambda},\,\mathrm{lr})$ & $(E_{LBFGS},\,\mathrm{lr})$ & $(E_{dn},\,\mathrm{lr}_{dn})$ & $\omega_{\max}$ \\
\hline
4.1 & 3 & 4500  & 6  & $(500,\,0.005)$ / $(500,\,0.005)$ & $(150,\,0.5)$ & $(5000,\,0.001)$ & $30\pi$ \\
4.2 & 2 & 4000  & 4  & $(500,\,0.005)$ / $(800,\,0.02)$  & $(50,\,0.1)$  & $(5000,\,0.001)$ & $30\pi$ \\
4.3 & 3 & 4500  & 9  & $(500,\,0.005)$ / $(500,\,0.005)$ & $(300,\,0.2)$ & $(2500,\,0.01)$  & $30\pi$ \\
4.4 & 7 & 4500  & 7  & $(500,\,0.005)$ / $(800,\,0.005)$ & $(500,\,1.0)$ & $(5000,\,0.005)$ & $30\pi$ \\
4.5 & 2 & 15000 & 16 & $(500,\,0.005)$ / $(500,\,0.005)$ & $(100,\,0.1)$ & $(5000,\,0.001)$ & $30\pi$ \\
\hline
\end{tabular}%
}
\end{table}


To quantitatively evaluate the reconstruction quality, we utilize the pointwise absolute error $|u - u^{\dagger}|$ and the relative $L^2$ error computed over $M_{\mathrm{test}}$ testing points:
\begin{equation}
    \text{Relative } L^2 \text{ Error} = \frac{\sqrt{\sum_{j=1}^{M_{\mathrm{test}}} (u_{j} - u^{\dagger}_{j})^2}}{\sqrt{\sum_{j=1}^{M_{\mathrm{test}}} (u^{\dagger}_{j})^2}},
\end{equation}
where $u$ and $u^{\dagger}$ denote the predicted and exact solutions, respectively.

\subsection{Inverse source problem}
\begin{example}\label{source_example}(\cite{duan2024recovering} Example 4.1)
    We set $\Omega=(0,1)^2$, $q(x) \equiv 1$ and $b(x) \equiv 1$. The boundary conditions are Neumann. The ground truth source term $f^{\dagger}(x)$ is constructed as a superposition of two Gaussian functions. Let $x = (x_1, x_2)$, then $f^{\dagger}$ is given by:
\begin{equation}
\begin{cases}

G_{1}(x_{1}, x_{2}) = \exp(-9 \times (x_{1}-0.3)^{2} - 25 \times (x_{2}-0.7)^{2}) \\

G_{2}(x_{1}, x_{2}) = \exp(-25 \times (x_{1}-0.7)^{2} - 9 \times (x_{2}-0.3)^{2}) \\

f^{\dagger}(x_{1}, x_{2}) = 25 G_{1}(x_{1}, x_{2}) + 36 G_{2}(x_{1}, x_{2})
\end{cases}
\end{equation}
The ground truth $u^{\dagger}$ is the solution to the aforementioned governing equation, subject to the zero boundary condition $u|_{\partial\Omega} = 0$. For the inverse problem, we assume the source term is unknown and must be recovered from noisy discrete measurements. The observed dataset, denoted by $\{(y_{i}^{\delta}, \mathbf{z}_{i}^{\delta})\}_{i=1}^{m}$, comprises both solution values and gradient vectors. These measurements are modeled as ground truth values corrupted by independent additive white Gaussian noise:

$$y_{i}^{\delta} = u^{\dagger}(\boldsymbol{x}_{i}) + \epsilon_{i}, \quad \mathbf{z}_{i}^{\delta} = \nabla u^{\dagger}(\boldsymbol{x}_{i}) + \boldsymbol{\eta}_{i},$$
where the noise terms follow normal distributions $\epsilon_{i} \sim \mathcal{N}(0, \sigma_{u}^{2})$ and $\boldsymbol{\eta}_{i} \sim \mathcal{N}(\mathbf{0}, \sigma_{\nabla u}^{2}\mathbf{I})$, respectively.

\end{example}

\begin{figure}[ht!]
    \centering
    \begin{subfigure}[b]{0.24\textwidth}
        \centering
        \includegraphics[width=\textwidth]{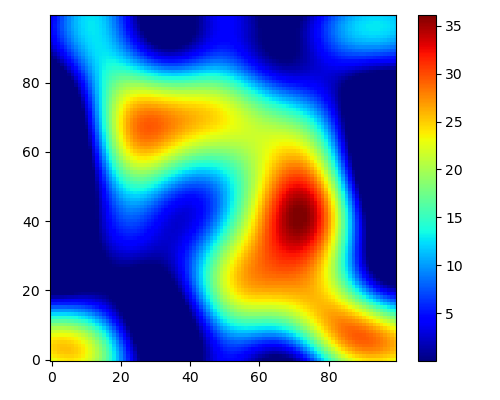}
        \caption{Stage 1}
        \label{source:stage1}
    \end{subfigure}
    \hfill 
    \begin{subfigure}[b]{0.24\textwidth}
        \centering
        \includegraphics[width=\textwidth]{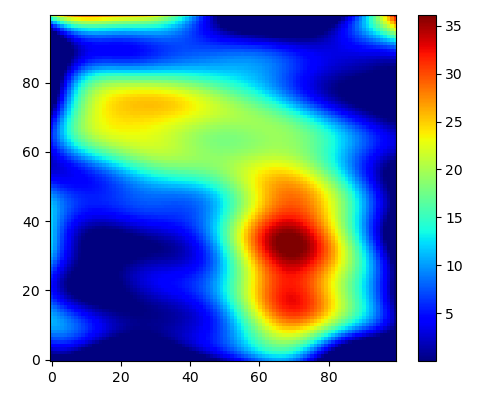}
        \caption{Stage 2}
        \label{source:stage2}
    \end{subfigure}
    \hfill
    \begin{subfigure}[b]{0.24\textwidth}
        \centering
        \includegraphics[width=\textwidth]{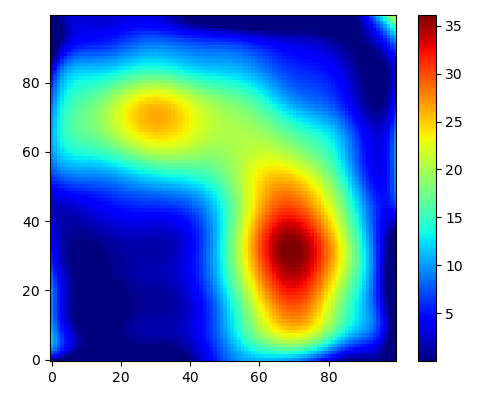}
        \caption{Stage 3}
        \label{source:stage3}
    \end{subfigure}
    \hfill
    \begin{subfigure}[b]{0.24\textwidth}
        \centering
        \includegraphics[width=\textwidth]{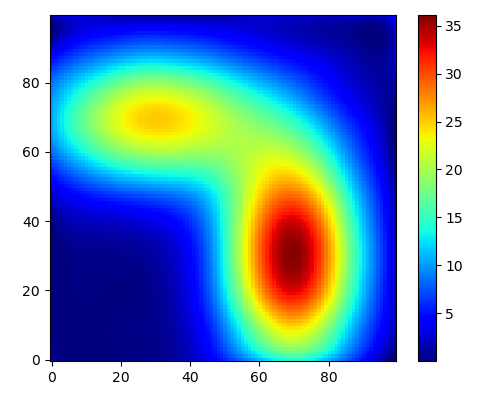}
        \caption{Stage 4}
        \label{source:stage4}
    \end{subfigure}
    
    \caption{Reconstructed $f$ at different training stages for Example \ref{source_example} with 1\% noise.}
    \label{source:four_stages}
\end{figure}

\begin{table}[ht!]
    \centering
    \caption{Relative $L^2$ errors of the reconstructed $f$ and $u$ for Example \ref{source_example} with 1\% noise.}
    \begin{tabular}{ccccccc} 
        \toprule
        Stage & 1 & 2 & 3 & 4 & 5 & 6 \\
        \midrule
        $err(f)$ & $5.66\times10^{-1}$ & $2.99\times10^{-1}$ & $1.37\times10^{-1}$ & $2.31\times10^{-2}$ & $1.09\times10^{-2}$ & $8.59\times10^{-3}$\\
        $err(u)$ & $6.13\times10^{-2}$ & $7.45\times10^{-3}$ & $2.73\times10^{-3}$ & $5.66\times10^{-4}$ & $2.87\times10^{-4}$ & $1.73\times10^{-4}$\\
        \bottomrule
    \end{tabular}
    \label{source_table1}
\end{table}

Figure \ref{source:four_stages} illustrates the evolution of the reconstructed source term $f$ across different training stages under 1\% noise. It is evident that the dominant structure is captured as early as the first stage, while fine-scale details are effectively resolved by the fourth stage. A quantitative assessment is provided in Table \ref{source_table1}, which lists the relative $L^2$ errors for both the solution $u$ and the source $f$. We observe a rapid error decay during the stages, followed by a deceleration in convergence. This behavior is attributed to the prioritized learning of low-frequency components, aligning with our theoretical analysis.  

\begin{figure}[ht!]
    \centering
    \begin{subfigure}[b]{0.24\textwidth}
        \centering
        \includegraphics[width=\textwidth]{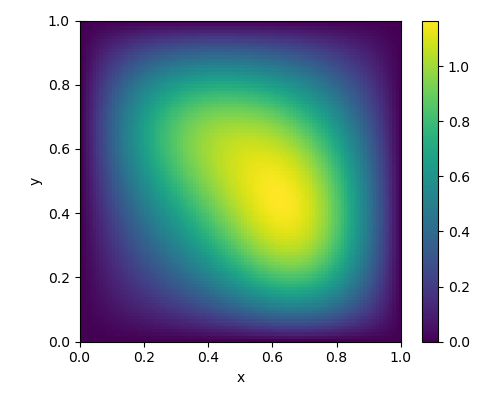}
        \caption{$u^{\dagger}$}
    \end{subfigure}
    \begin{subfigure}[b]{0.24\textwidth}
        \centering
        \includegraphics[width=\textwidth]{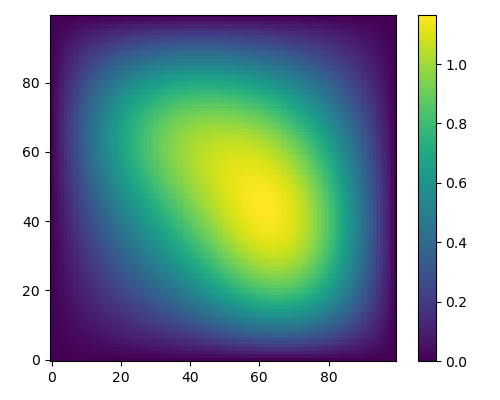}
        \caption{$u$}
    \end{subfigure}
    \begin{subfigure}[b]{0.24\textwidth}
        \centering
        \includegraphics[width=\textwidth]{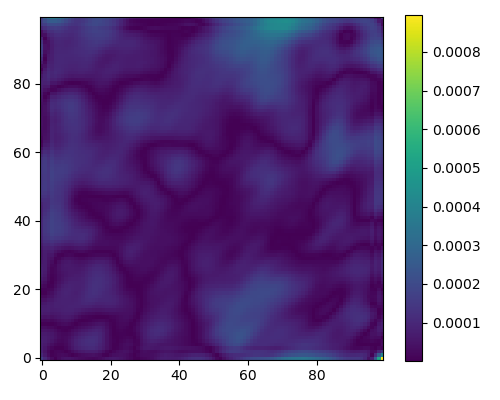}
        \caption{$|u - u^{\dagger}|$}
    \end{subfigure}

    \vspace{1em}

    \begin{subfigure}[b]{0.24\textwidth}
        \centering
        \includegraphics[width=\textwidth]{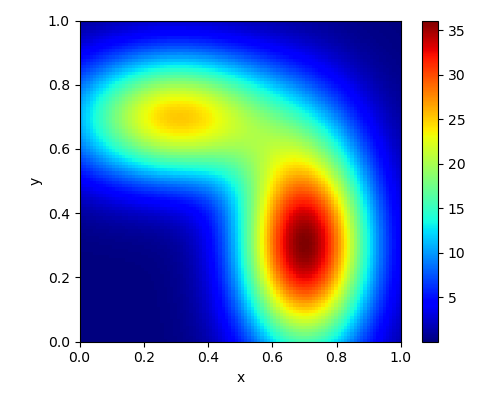}
        \caption{$f^{\dagger}$}
    \end{subfigure}
    \begin{subfigure}[b]{0.24\textwidth}
        \centering
        \includegraphics[width=\textwidth]{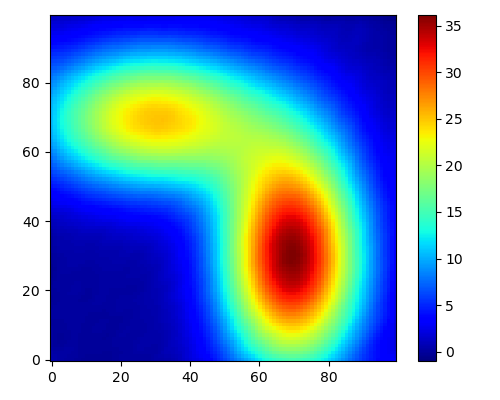}
        \caption{$f$}
    \end{subfigure}
    \begin{subfigure}[b]{0.24\textwidth}
        \centering
        \includegraphics[width=\textwidth]{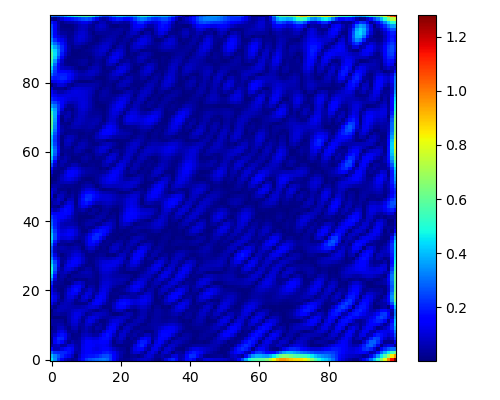}
        \caption{$|f - f^{\dagger}|$}
    \end{subfigure}

    \caption{Reconstruction results for Example \ref{source_example} with 1\% noise.}
    \label{source:result}
\end{figure}
Figure \ref{source:result} illustrates the reconstructed solution, the ground truth and the absolute error at last training stages. The remaining errors in the reconstructed source term $f$ are primarily concentrated near the boundaries, which can be attributed to the inherent challenges in estimating second-order derivatives given the function value itself.

\begin{table}[ht!]
    \centering
    \caption{$L^2$ relative errors for Example \ref{source_example} without observation smoothing.}
    \begin{tabular}{ccccc} 
        \toprule
        Method & \multicolumn{4}{c}{Ours without smoothing}  \\
        
        \cmidrule(lr){2-5} 
        
        $\delta$ & 1\% & 2\% & 10\% & 20\% \\
        \midrule
        
        $err(f)$ & $8.59\times10^{-3}$ & $1.57\times10^{-2}$ & $9.22\times10^{-2}$ & $1.48\times10^{-1}$  \\
        $err(u)$ & $1.73\times10^{-4}$ & $3.59\times10^{-4}$ & $1.82\times10^{-3}$ & $3.73\times10^{-3}$ \\
        \bottomrule
    \end{tabular}
    \label{source_without}
\end{table}

\begin{table}[ht!]
    \centering
    \caption{Relative $L^2$ errors and computational costs for reconstructing $f$ across different methods.}
    \begin{tabular}{ccccc} 
        \toprule
        \multirow{1}{*}{Error} & \multicolumn{3}{c}{$f$} & \multirow{2}{*}{Times} \\
        \cmidrule(lr){2-4}
        
        $\delta$ & 1\% & 10\% & 20\% & \\ 
        \midrule
        
        ALBC  & $8.59\times10^{-3}$ & $4.33\times10^{-2}$ & $5.26\times10^{-2}$ & 12.87/44.30 \\
        PINNs & $6.95\times10^{-2}$ & $7.04\times10^{-2}$ & $7.16\times10^{-2}$ & 253.81 \\
        L-ALBC  & $1.28\times10^{-2}$ & $4.33\times10^{-2}$ & $5.29\times10^{-2}$ & 12.1/43.49 \\
        Collocation  & $1.08\times10^{-1}$ & $1.83\times10^{-1}$ & $1.84\times10^{-1}$ & 136 \\
        \bottomrule
    \end{tabular}
    \label{source_with_noise}
\end{table}

Table \ref{source_without} presents the relative $L^2$ errors for $u$ and $f$ across various noise levels in the absence of the observation smoothing step. The reconstruction error exhibits a near-linear dependence on the noise level, empirically validating our theoretical analysis that the error scales as $\mathcal{O}(\delta)$. 

To further demonstrate the superiority of our approach, we compare ALBC against the standard PINN baseline and the classical collocation method under varying noise conditions in Table \ref{source_with_noise}. For the PINN baseline, we adopt the network architecture detailed in \cite{duan2024current}, utilizing $10,000$ collocation points and training for $50,000$ epochs. For the classical collocation method, we employ radial basis functions (RBFs) with $10,000$ collocation points and $8,100$ basis functions. Furthermore, to validate the effectiveness of the global fine-tuning step, we introduce a ``lite'' variant of our method, denoted as L-ALBC. This variant omits the periodic fine-tuning phase (step 4 in Algorithm~\ref{main_algo}). The reported times in Table \ref{source_with_noise} for ALBC and L-ALBC follow a ``without smoothing / with smoothing'' format. The first value represents the pure training time applied in low-noise scenarios, while the second value includes the additional computational overhead of the neural network-based denoising pre-processing required for high-noise cases. As clearly indicated in the results, ALBC consistently outperforms all baseline methods across all tested noise levels.

\begin{figure}[ht!]
    \centering
    \includegraphics[width=0.8\linewidth]{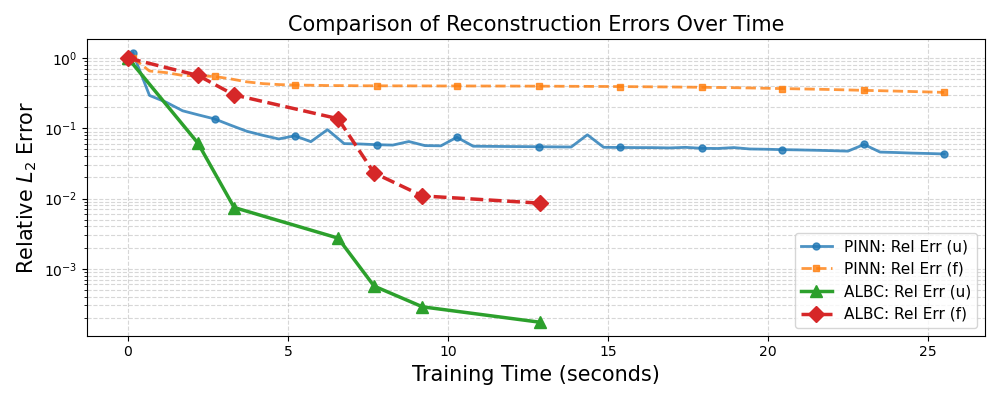}
    \caption{Evolution of relative $L^2$ errors for $u$ and $f$ versus training time for ALBC and PINN methods. We demonstrate the result for Example \ref{source_example} with 1\% noise.}
    \label{source:compare}
\end{figure}

    

Finally, Figure \ref{source:compare} illustrates the evolution of the relative error with respect to training time, comparing our proposed method against the PINN baseline. As demonstrated, our approach achieves significantly faster convergence in the reconstruction of both $u$ and $f$.


\subsection{Inverse potential problem}
\begin{example}\label{potential_example}
    We aim to recover the potential coefficient $b$. The domain, boundary conditions, and observation data are consistent with Example \ref{source_example}, with $q(x)\equiv 1$. The true solution $u = 1 + \sin(x)\sin(y)$ satisfies $|u^\dagger| \ge 1 > 0$ on $\bar{\Omega}$ (verifying the non-degeneracy condition in Assumption~2). We set $b = 0.5 + \sin(\pi x)\sin(\pi y)$, while the source term $f$ is computed explicitly using the PDE.
\end{example}

\begin{figure}[ht!]
    \centering
    \begin{subfigure}[b]{0.24\textwidth}
        \centering
        \includegraphics[width=\textwidth]{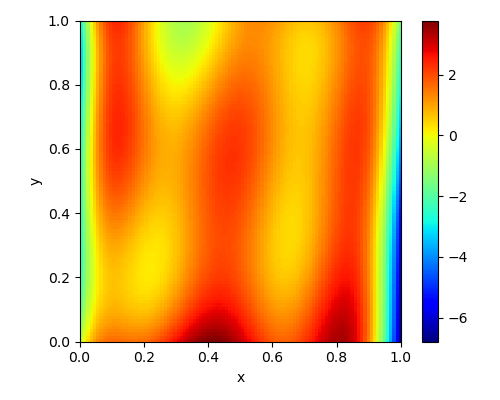}
        \caption{Stage 1}
    \end{subfigure}
    \hfill 
    \begin{subfigure}[b]{0.24\textwidth}
        \centering
        \includegraphics[width=\textwidth]{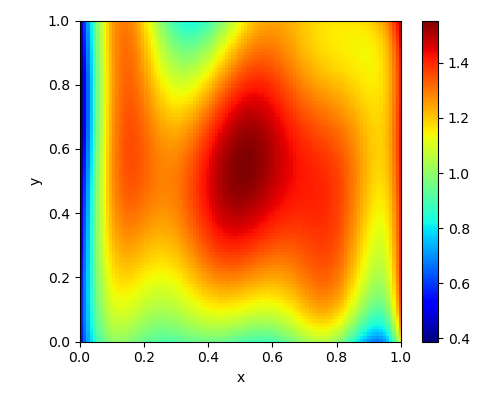}
        \caption{Stage 2}
    \end{subfigure}
    \hfill
    \begin{subfigure}[b]{0.24\textwidth}
        \centering
        \includegraphics[width=\textwidth]{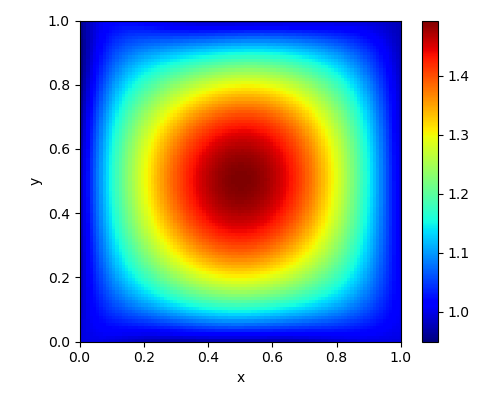}
        \caption{Stage 3}
    \end{subfigure}
    \hfill
    \begin{subfigure}[b]{0.24\textwidth}
        \centering
        \includegraphics[width=\textwidth]{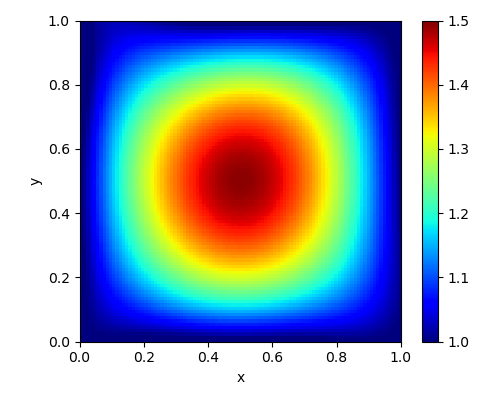}
        \caption{Stage 4}
    \end{subfigure}
    
    \caption{Reconstructed $b$ at different training stages for Example \ref{potential_example} with 1\% noise.}
    \label{potential:four_stages}
\end{figure}

\begin{table}[htbp]
    \centering
    \caption{Relative $L^2$ errors of the reconstructed $b$ and $u$ for Example \ref{potential_example} with 1\% noise.}
    \begin{tabular}{ccccccc} 
        \toprule
        Stage & 1 & 2 & 3 & 4 \\
        \midrule
        $err(b)$ & $9.21\times10^{-1}$ & $9.15\times10^{-2}$ & $6.28\times10^{-3}$ & $6.12\times10^{-3}$ \\
        $err(u)$ & $1.05\times10^{-2}$ & $6.94\times10^{-4}$ & $9.85\times10^{-5}$ & $6.09\times10^{-5}$\\
        \bottomrule
    \end{tabular}
    \label{potential_table1}
\end{table}

Figure \ref{potential:four_stages} illustrates the reconstructed potential $b$ at four distinct stages, and the pointwise absolute error of $u$ and $f$ is presented in Figure \ref{potential:result}. These results demonstrate the spectral bias where the neural network prioritizes capturing the global structure before refining high-frequency details. Table \ref{potential_table1} presents the relative $L^2$ errors for both the state $u$ and the potential $b$ under a $1\%$ noise level. The progressive decrease in error across stages confirms that our method achieves high accuracy and stability in low-noise scenarios.

\begin{figure}[ht!]
    \centering
    \begin{subfigure}[b]{0.24\textwidth}
        \centering
        \includegraphics[width=\textwidth]{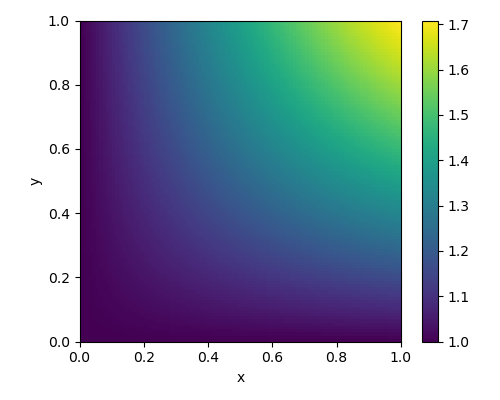}
        \caption{$u^{\dagger}$}
    \end{subfigure}
    \begin{subfigure}[b]{0.24\textwidth}
        \centering
        \includegraphics[width=\textwidth]{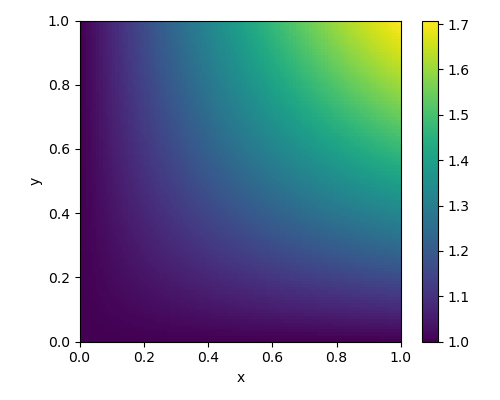}
        \caption{$u$}
    \end{subfigure}
    \begin{subfigure}[b]{0.24\textwidth}
        \centering
        \includegraphics[width=\textwidth]{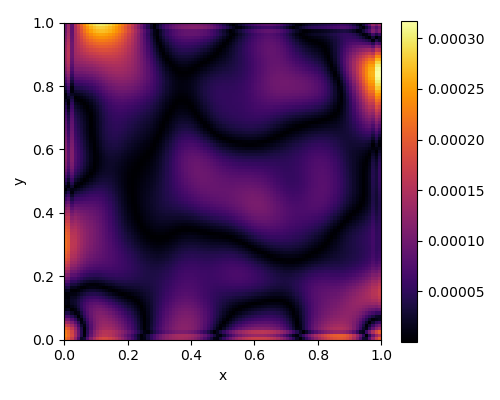}
        \caption{$|u - u^{\dagger}|$}
    \end{subfigure}

    \vspace{1em}

    \begin{subfigure}[b]{0.24\textwidth}
        \centering
        \includegraphics[width=\textwidth]{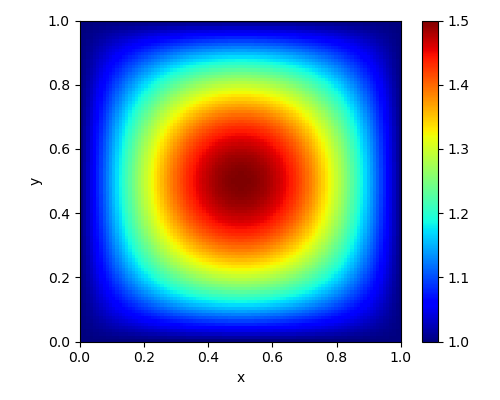}
        \caption{$b^{\dagger}$}
    \end{subfigure}
    \begin{subfigure}[b]{0.24\textwidth}
        \centering
        \includegraphics[width=\textwidth]{fig/potential/potential/f.png}
        \caption{$b$}
    \end{subfigure}
    \begin{subfigure}[b]{0.24\textwidth}
        \centering
        \includegraphics[width=\textwidth]{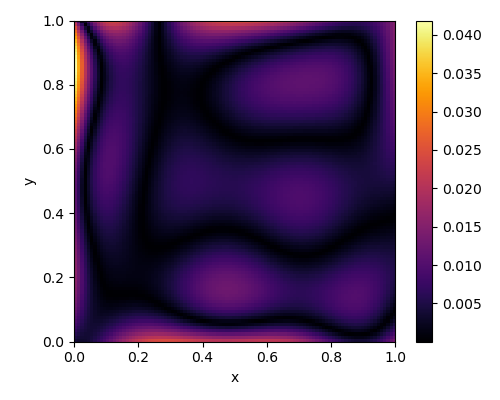}
        \caption{$|b - b^{\dagger}|$}
    \end{subfigure}

    \caption{Reconstruction results for Example \ref{potential_example} with 1\% noise.}
    \label{potential:result}
\end{figure}

\begin{table}[ht!]
    \centering
    \caption{Relative $L^2$ errors of Example \ref{potential_example} across different methods.}
    \begin{tabular}{ccccccccc} 
        \toprule
        Error & \multicolumn{3}{c}{$b$} & \multicolumn{3}{c}{$u$} \\

        \cmidrule(lr){2-4} \cmidrule(lr){5-7}

        $\delta$ & 20\% & 30\% & 50\% & 20\% & 30\% & 50\% \\
        \midrule

        ALBC & $1.46\times10^{-3}$ & $1.69\times10^{-2}$ & $1.59\times10^{-2}$ & $6.47\times10^{-4}$ & $7.12\times10^{-4}$ & $9.71\times10^{-4}$ \\
        PINNs & $1.61\times10^{-2}$ & $1.69\times10^{-2}$ & $1.96\times10^{-2}$  & $7.14\times10^{-4}$ & $9.67\times10^{-4}$ & $1.81\times10^{-3}$ \\
        \bottomrule
    \end{tabular}
    \label{potential_with_noise}
\end{table}

Furthermore, we examine the reconstruction performance under high-noise regimes. Table \ref{potential_with_noise} compares the accuracy of our proposed method with the PINNs baseline across various high-noise scenarios. The results reveal that our approach consistently achieves superior reconstruction fidelity.

Finally, Figure~\ref{potential:noise_loglog} displays the relative $L^2$ errors of $u$ and $b$ as a function of $\delta_{noise}$ on a log--log scale. Since $b$ is smooth ($2(p-4)/(p-2)\to2$ as $p\to\infty$), Theorem~\ref{thm:balance} predicts that both the state and parameter errors will decay at the rate of $\mathcal{O}(\delta_{noise})$. Remarkably, both error curves align well with straight lines possessing a slope of approximately $1$, perfectly confirming the theoretical predictions.

\begin{figure}[ht!]
    \centering
    \includegraphics[width=0.7\textwidth]{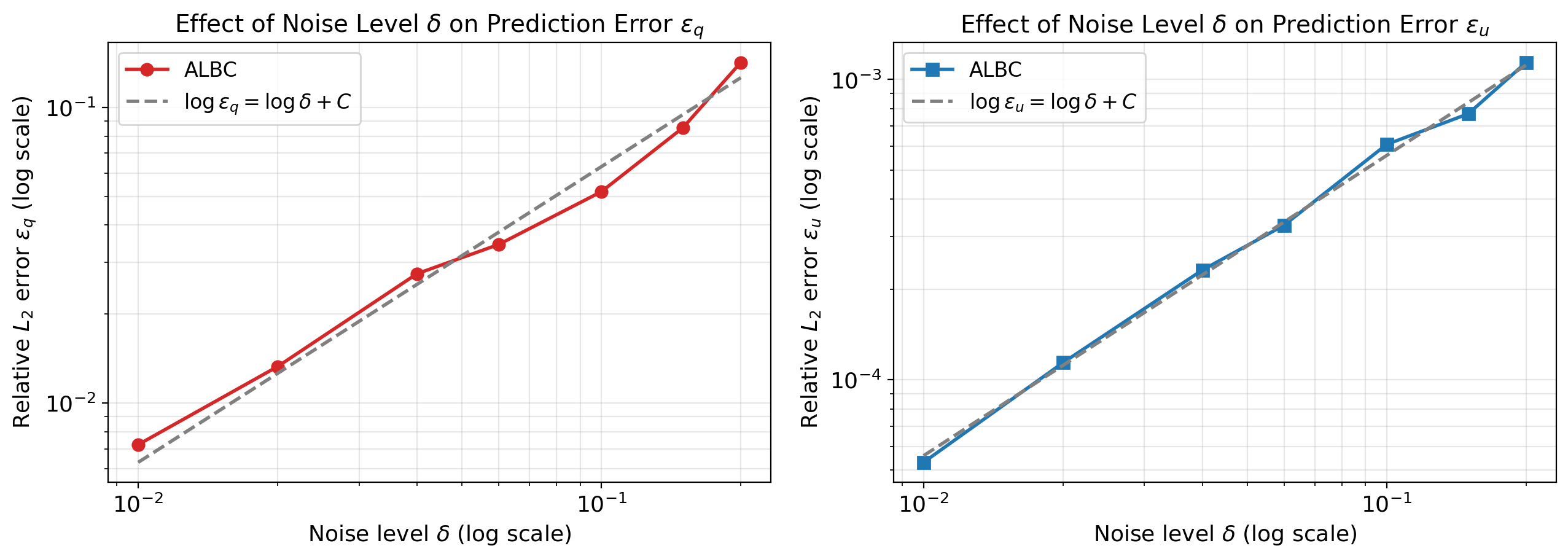}
    \caption{Relative $L^2$ errors of $u$ and $b$ versus the noise level $\delta_{noise}$ (log--log) for Example \ref{potential_example}; both slopes $\approx1$ match the $\mathcal{O}(\delta_{noise})$ rates.}
    \label{potential:noise_loglog}
\end{figure}

\subsection{Identification of diffusion coefficient}
\begin{example}\label{diffusion_example}
     (\cite{jin2024conductivity} Example 5.1) We set $\Omega=(-1,1)^2$, $b(x)\equiv0$ and $\mathcal{B}(u)=\frac{\partial u}{\partial \mathbf{n}}$ (Neumann boundary condition). Our objective is to recover the unknown conductivity $q(x)$ from noisy internal gradient measurements $\nabla z^\delta$. For the numerical simulation, we set the ground truth conductivity $q^\dagger$ as: $$q^\dagger(x) = 1 + \sum_{i=1}^3 s_i(x),$$where the component functions are defined as:
\begin{align*}
    s_1(x) &= 0.3 \exp\left(-20(x_1 - 0.3)^2 - 15(x_2 - 0.3)^2\right), \\
    s_2(x) &= -0.3 \exp\left(-10x_1^2 - 10(x_2 + 0.5)^2\right), \\
    s_3(x) &= 0.2 \exp\left(-15(x_1 + 0.4)^2 - 15(x_2 - 0.35)^2\right).
\end{align*}
The exact potential is chosen as the polynomial $u^\dagger(x) = x_1 + x_2 + \frac{1}{3}(x_1^3 + x_2^3)$, whose gradient $\nabla u^\dagger = (1+x_1^2, 1+x_2^2)$ satisfies $|\nabla u^\dagger| \ge \sqrt{2} > 0$ on $\bar{\Omega}$, ensuring the non-degeneracy condition in Assumption~2. The corresponding source term $f$ and boundary flux $g$ are derived by substituting $q^\dagger$ and $u^\dagger$ into the governing equation \eqref{eq:elliptic_system}. The observational data $\nabla z^\delta$ is generated by adding pointwise Gaussian noise to the exact gradient:
\begin{equation}\label{eq:grad_obs}
    \nabla z^\delta(x) = \nabla u^\dagger(x) + \delta \cdot \max_{y \in \Omega} \|\nabla u^\dagger(y)\|_{\infty} \cdot \iota(x),
\end{equation}
where $\delta$ is the relative noise level and $\iota(x) \sim \mathcal{N}(0, I)$ is standard Gaussian noise.
\end{example}

Figure \ref{diffusion:four_stages} and Table \ref{diffusion_table1} present the reconstructed $q$ and the corresponding accuracy across different stages under $1\%$ noise. These results indicate that the reconstruction of the coefficient $q$ consistently exhibits characteristic spectral bias and rapid initial convergence.

\begin{figure}[ht!]
    \centering
    \begin{subfigure}[b]{0.24\textwidth}
        \centering
        \includegraphics[width=\textwidth]{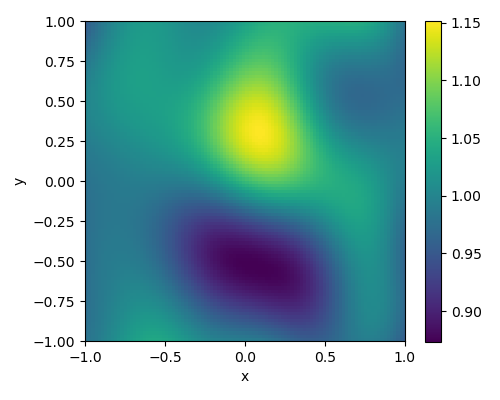}
        \caption{Stage 1}
        \label{diffusion:stage1}
    \end{subfigure}
    \hfill 
    \begin{subfigure}[b]{0.24\textwidth}
        \centering
        \includegraphics[width=\textwidth]{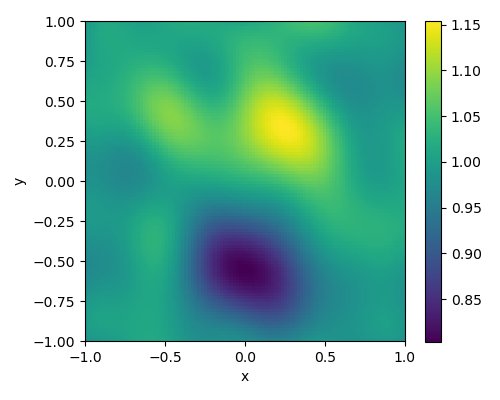}
        \caption{Stage 2}
        \label{diffusion:stage2}
    \end{subfigure}
    \hfill
    \begin{subfigure}[b]{0.24\textwidth}
        \centering
        \includegraphics[width=\textwidth]{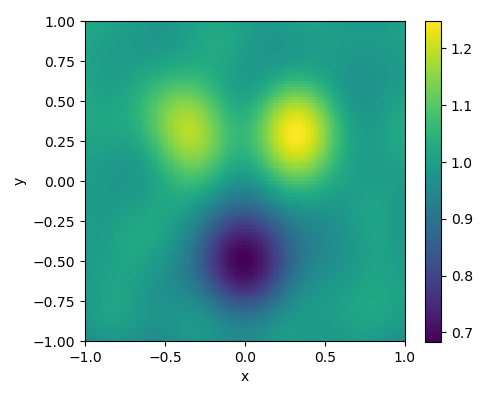}
        \caption{Stage 3}
        \label{diffusion:stage3}
    \end{subfigure}
    \hfill
    \begin{subfigure}[b]{0.24\textwidth}
        \centering
        \includegraphics[width=\textwidth]{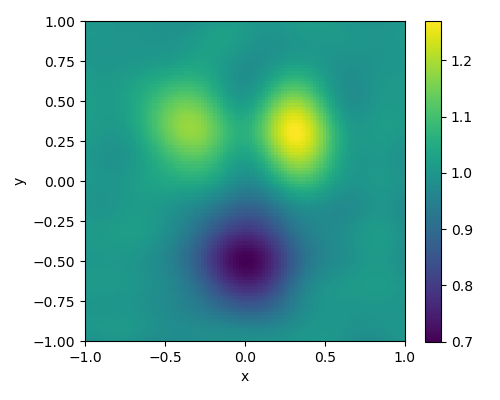}
        \caption{Stage 4}
        \label{diffusion:stage4}
    \end{subfigure}
    
    \caption{Reconstructed $q$ at different training stages for Example \ref{diffusion_example} under 1\% noise level.}
    \label{diffusion:four_stages}
\end{figure}

\begin{table}[ht!]
    \centering
    \caption{Relative $L^2$ errors for Example \ref{diffusion_example} at different training stages with 1\% noise.}
    \begin{tabular}{ccccccc} 
        \toprule
        Stage & 1 & 2 & 3 & 4 & 5 & 6 \\
        \midrule
        $err(q)$ & $5.07\times10^{-2}$ & $3.62\times10^{-2}$ & $1.51\times10^{-2}$ & $1.20\times10^{-2}$ & $1.06\times10^{-2}$ & $9.80\times10^{-3}$\\
        \bottomrule
    \end{tabular}
    \label{diffusion_table1}
\end{table}

Figure \ref{diffusion:error} displays the exact coefficient $q$, the reconstructed coefficient, and the corresponding absolute error map at the final training stage under $1\%$ noise level. The error distribution exhibits an oscillatory pattern, which validates the necessity of our spectral-based initialization and adaptive sampling strategies, while also indicating room for further refinement. For a comprehensive comparison, Table \ref{diffusion_with_noise} summarizes the accuracy and computational costs of our method alongside the Mixed DNN method \cite{jin2024conductivity} across various noise levels. The results demonstrate that ALBC consistently achieves higher precision while requiring significantly less computational time than the Mixed DNN.

\begin{figure}[ht!]
    \centering
    \begin{subfigure}[b]{0.24\textwidth}
        \centering
        \includegraphics[width=\textwidth]{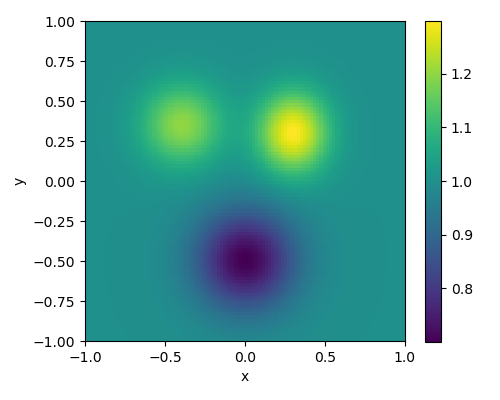}
        \caption{$q$}
    \end{subfigure}
    \begin{subfigure}[b]{0.24\textwidth}
        \centering
        \includegraphics[width=\textwidth]{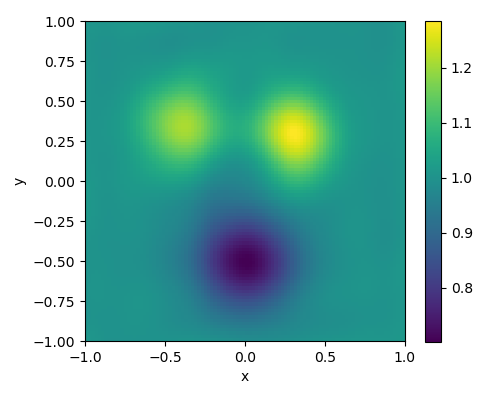}
        \caption{$q^{\dagger}$}
    \end{subfigure}
    \begin{subfigure}[b]{0.24\textwidth}
        \centering
        \includegraphics[width=\textwidth]{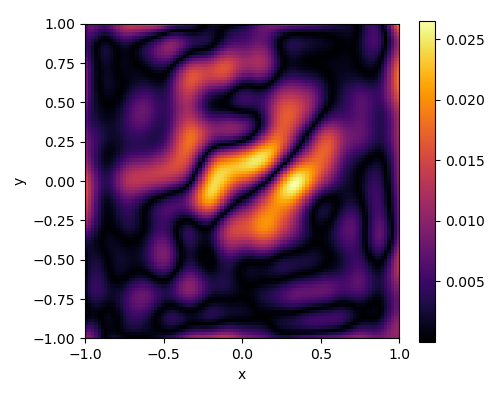}
        \caption{$|q-q^{\dagger}|$}
    \end{subfigure}

    \caption{Reconstruction results for Example \ref{diffusion_example} with 1\% noise.}
    \label{diffusion:error}
\end{figure}

\begin{table}[th!]
    \centering
    \caption{Relative $L^2$ errors for Example \ref{diffusion_example}across different methods.}
    \begin{tabular}{ccccc} 
        \toprule
        Method & \multicolumn{3}{c}{Error} & \multirow{2}{*}{Times} \\
        \cmidrule(lr){2-4}
        
        $\delta$ & 0\% & 1\% & 10\% & \\ 
        \midrule
        
        ALBC & $7.67\times10^{-3}$ & $7.77\times10^{-3}$ & $1.97\times10^{-2}$ & 88.4 \\
        Mixed DNN & $1.01\times10^{-2}$ & $8.52\times10^{-3}$ & $4.30\times10^{-2}$ & 374.4 \\ 
        \bottomrule
    \end{tabular}
    \label{diffusion_with_noise}
\end{table}

\subsection{Current Density Impedance Imaging}

Although the theoretical analysis in Section \ref{sec_analysis} focuses on linear observation operators, our alternating collocation framework naturally accommodates bilinear or nonlinear observations. By decoupling the joint inversion, the observation nonlinearity is effectively bypassed at each individual step. As a representative example, we consider Current Density Impedance Imaging (CDII), where the observation $|J| = q|\nabla u|$ is nonlinear with respect to $(q, u)$ but conditionally linear in $q$ when $u$ is fixed. Exploiting this property, we prioritize augmenting the basis for $q$. Specifically, at the $k$-th stage of ALBC, we first update the parameter basis $\psi_k$ using the surrogate observation $\frac{|J|}{|\nabla u_{k-1}|}$ alongside the PDE residual. Subsequently, we update the state basis $\phi_k$ by minimizing \eqref{u_unnorm}, omitting the data fidelity term $\mathcal{R}_{\mathrm{data}}^{(k)}$ during this step.

\begin{example}\label{eit_example}(\cite{nachman2009recovering}, \cite{duan2024current} Example 5.1)
    We consider the Current Density Impedance Imaging (CDII) problem\cite{nachman2009recovering}: $\nabla \cdot (q \nabla u) = 0$ in $\Omega$, $u = g$ on $\partial\Omega$, where $q$ is the unknown conductivity (with $b \equiv 0$ and $f \equiv 0$). The observation data is the current density magnitude $|J| = q |\nabla u|$. The ground truth conductivity is $q^\dagger = 1 + 0.3(G_1 - G_2 - G_3)$, where
$$\begin{aligned}
G_1 &= 0.3(1 - 3\bar{x})^2 \exp\left[-9\bar{x}^2 - (6y-2)^2\right], \\
G_2 &= \left(\tfrac{3\bar{x}}{5} - 27\bar{x}^3 - (3(2y-1))^5\right) \exp\left[-9\bar{x}^2 - 9(2y-1)^2\right], \\
G_3 &= \exp\left[-(3\bar{x}+1)^2 - 9(2y-1)^2\right],
\end{aligned}$$
with $\bar{x} = 2x - 1$.
\end{example}

We initially evaluate the proposed method by recovering the potential $u$ and conductivity $q$ from data containing 2\% noise. Figure \ref{eit:four_stages} illustrates the conductivity reconstructions across various stages. The results demonstrate that while a rough approximation is obtainable after only a few stages, increasing the stage count is essential for capturing intricate structural details. Table \ref{eit_table_stage} presents the relative $L^2$ errors of the potential and conductivity at different stages. We observe that the reconstruction accuracy for both $q$ and $u$ improves progressively as the stages increase. 

Figure \ref{fig:eit_error} illustrates the predicted values and error maps for $q$ and $u$ at Stage 7. It is evident from the figure that the error distributions for both potential and conductivity exhibit distinct wave-like patterns. This observation further validates the rationality of our frequency-based initialization and sampling strategy.

\begin{table}[ht!]
    \centering
    \caption{Relative $L^2$ errors for reconstructing $q$ and $u$ for Example \ref{eit_example} with 2\% noise.}
    \begin{tabular}{ccccccc} 
        \toprule
        Stage & 1 & 2 & 3 & 4 & 5 & 6\\
        \midrule
        $err(q)$ & $4.7\times10^{-2}$ & $4.2\times10^{-2}$ & $3.5\times10^{-2}$ & $3.1\times10^{-2}$& $2.9\times10^{-2}$ & $2.7\times10^{-2}$ \\
        $err(u)$ & $8.0\times10^{-3}$ & $5.8\times10^{-3}$ & $4.5\times10^{-3}$ &  $4.0\times10^{-3}$ &$3.3\times10^{-3}$ & $3.0\times10^{-3}$ \\
        \bottomrule
    \end{tabular}
    \label{eit_table_stage}
\end{table}

\begin{figure}[ht!]
    \centering
    \begin{subfigure}[b]{0.24\textwidth}
        \centering
        \includegraphics[width=\textwidth]{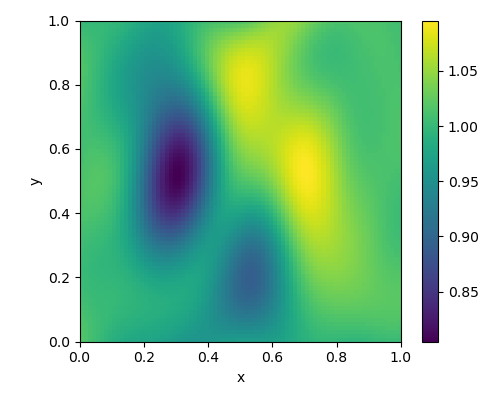}
        \caption{Stage=1}
        \label{eit:stage1}
    \end{subfigure}
    \hfill 
    \begin{subfigure}[b]{0.24\textwidth}
        \centering
        \includegraphics[width=\textwidth]{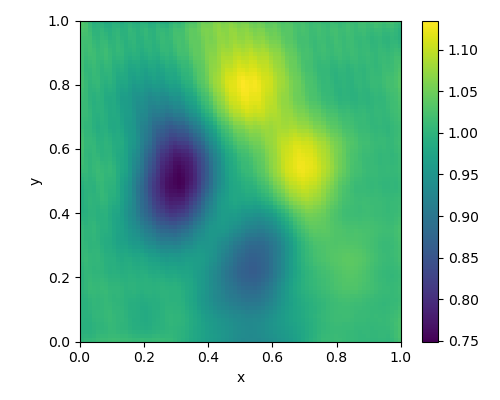}
        \caption{Stage=3}
        \label{eit:stage2}
    \end{subfigure}
    \hfill
    \begin{subfigure}[b]{0.24\textwidth}
        \centering
        \includegraphics[width=\textwidth]{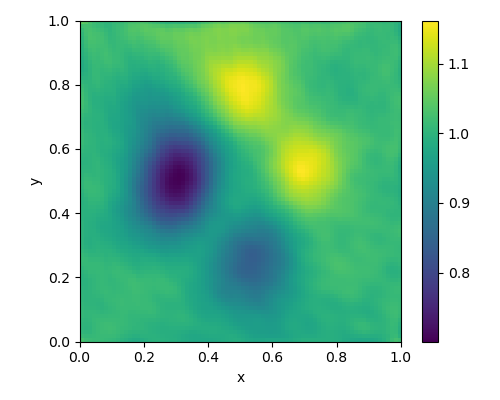}
        \caption{Stage=5}
        \label{eit:stage3}
    \end{subfigure}
    \hfill
    \begin{subfigure}[b]{0.24\textwidth}
        \centering
        \includegraphics[width=\textwidth]{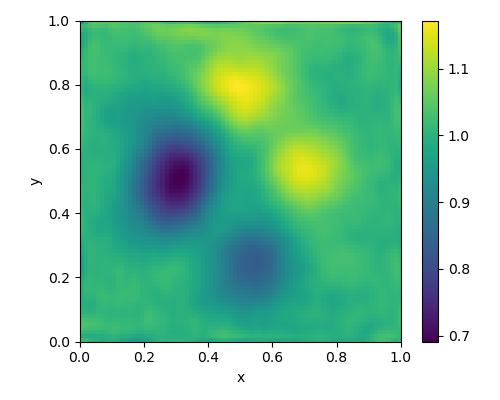}
        \caption{Stage=7}
        \label{eit:stage4}
    \end{subfigure}
    
    \caption{The reconstructed conductivity at different stages for Example \ref{eit_example} with 2\% noise.}
    \label{eit:four_stages}
\end{figure}

To evaluate the performance of our proposed method, we compare it against the Weighted Least Gradient Method (WLGM) \cite{nachman2009recovering}. Table \ref{eit_table_noise} presents the relative errors of the potential and conductivity reconstructed from current density magnitude data under varying noise levels. The results indicate that while the relative errors for the potential $u$ are comparable between the two methods, our approach significantly outperforms WLGM in conductivity reconstruction. Furthermore, our method exhibits remarkable robustness, maintaining its advantage particularly as the noise level increases.

\begin{figure}[ht!]
    \centering
    \begin{subfigure}[b]{0.24\textwidth}
        \centering
        \includegraphics[width=\textwidth]{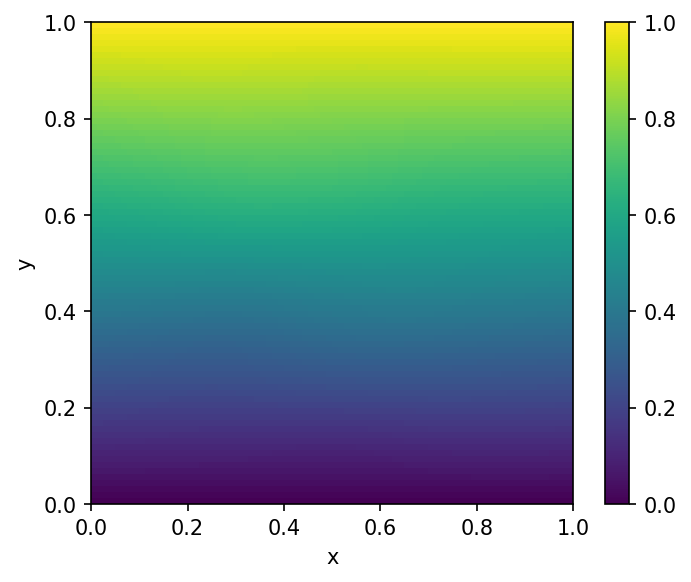}
        \caption{$u^{\dagger}$}
    \end{subfigure}
    \begin{subfigure}[b]{0.24\textwidth}
        \centering
        \includegraphics[width=\textwidth]{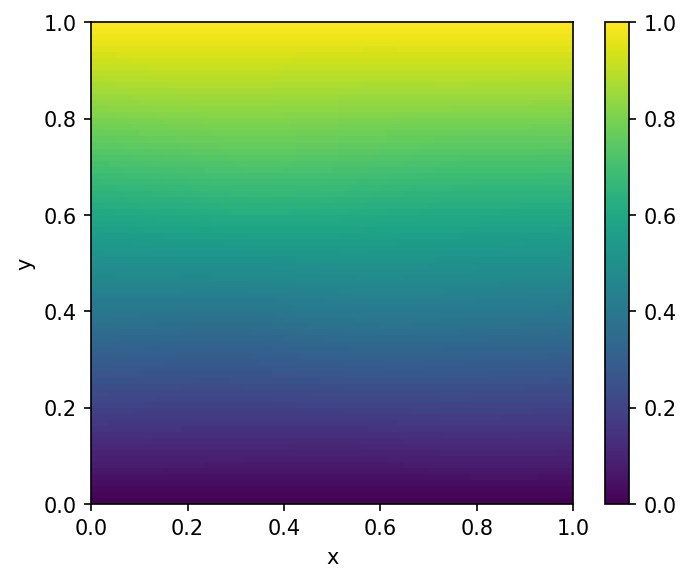}
        \caption{$u$}
    \end{subfigure}
    \begin{subfigure}[b]{0.24\textwidth}
        \centering
        \includegraphics[width=\textwidth]{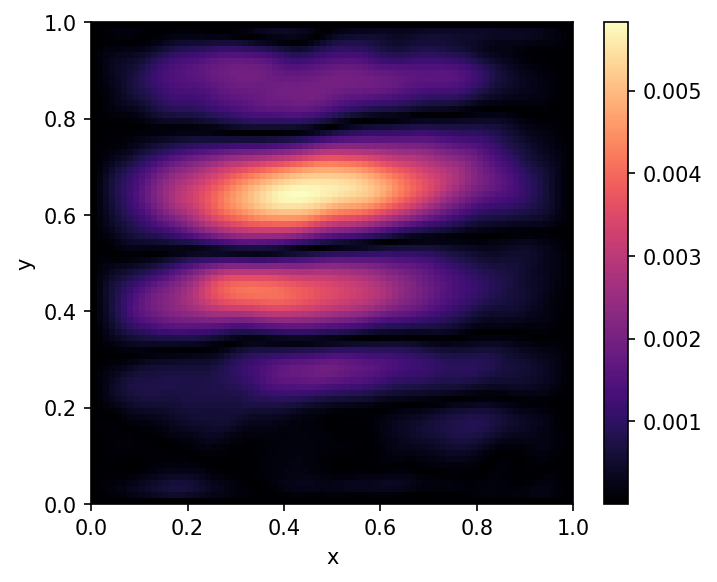}
        \caption{$|u - u^{\dagger}|$}
    \end{subfigure}

    \vspace{1em}

    \begin{subfigure}[b]{0.24\textwidth}
        \centering
        \includegraphics[width=\textwidth]{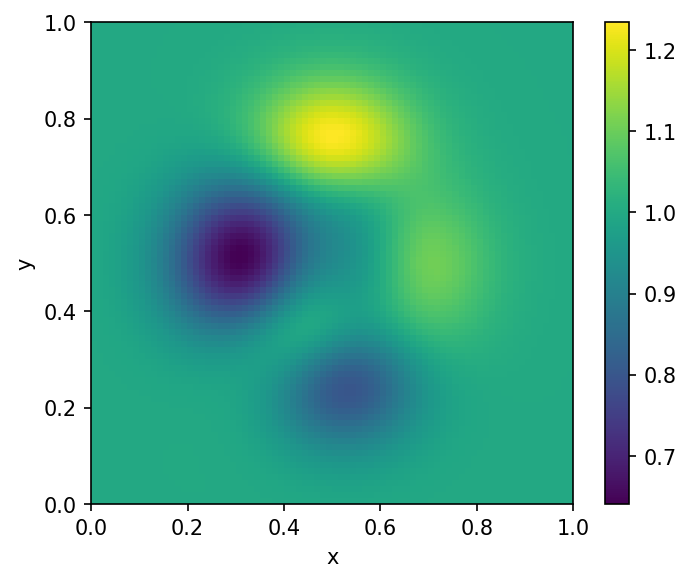}
        \caption{$q^{\dagger}$}
    \end{subfigure}
    \begin{subfigure}[b]{0.24\textwidth}
        \centering
        \includegraphics[width=\textwidth]{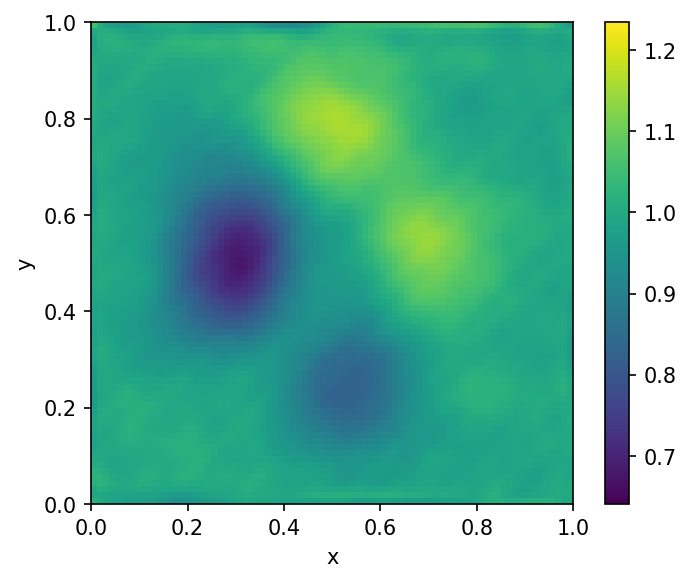}
        \caption{$q$}
    \end{subfigure}
    \begin{subfigure}[b]{0.24\textwidth}
        \centering
        \includegraphics[width=\textwidth]{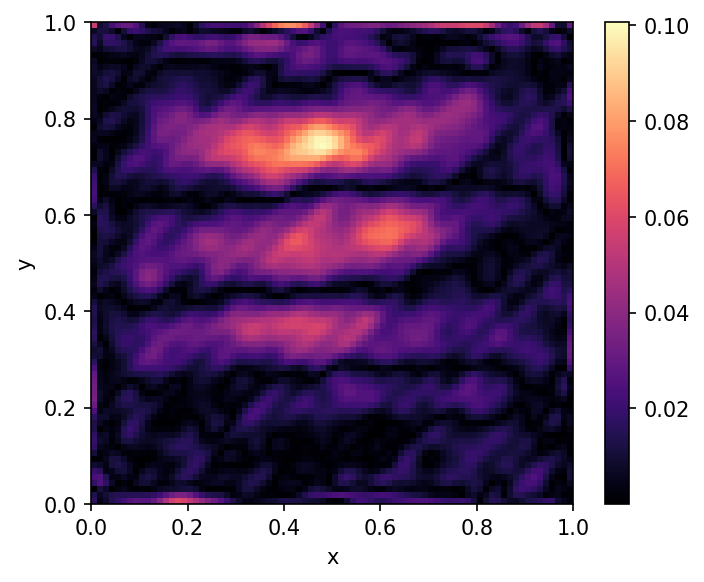}
        \caption{$|q - q^{\dagger}|$}
    \end{subfigure}

    \caption{Reconstruction results for Example \ref{eit_example} with 2\% noise.}
    \label{fig:eit_error}
\end{figure}

\begin{table}[ht!]
    \centering
    \caption{Relative $L_2$ errors for reconstructing $q$ and $u$ across different methods for Example \ref{eit_example}.}
    \begin{tabular}{ccccccc} 
        \toprule
        Method &  & ALBC &  &   & WLGM & \\
        \midrule
        $\delta$ & 0.1\% & 1\%  & 10\% & 0.1\% & 1\% & 10\% \\
        \midrule
        $err(q)$ & $2.37\times10^{-2}$ &  $2.47\times10^{-2}$ & $3.93\times10^{-2}$ &$5.2\times10^{-2}$ & $5.5\times10^{-2}$ & $1.4\times10^{-1}$ \\
        $err(u)$ & $2.3\times10^{-3}$ & $2.4\times10^{-3}$ & $4.2\times10^{-3}$ &$2.3\times10^{-3}$ & $2.5\times10^{-3}$ & $5.5\times10^{-2}$\\
        \bottomrule
    \end{tabular}
    \label{eit_table_noise}
\end{table}

\subsection{High-Dimensional Problem}
\begin{example}(\cite{jin2024conductivity} Example 5.10)
We now consider the inversion of the diffusion coefficient $q(\mathbf{x})$ in the 5D unit hypercube $\Omega = (0, 1)^5$. Setting $b(\mathbf{x})\equiv 0$, the system is subject to Dirichlet boundary conditions, with both the boundary values and the source term $f$ analytically derived from the exact solutions. Employing the observation operator defined in \eqref{eq:grad_obs}, the exact parameter $q^\dagger(\mathbf{x})$ and state variable $u^\dagger(\mathbf{x})$ are given by:
\begin{align*}
    q^\dagger(\mathbf{x}) &= 1 + 0.5(x_1 x_5 + x_2 x_4 + x_3^2) - 0.3\exp\left(-25(x_1-0.5)^2 - 25(x_2-0.5)^2\right), \\
    u^\dagger(\mathbf{x}) &= \sum_{i=1}^5 \left(x_i + \frac{1}{3}x_i^3\right).
\end{align*}
    \label{5D_example}
\end{example}

\begin{table}[ht]
    \centering
    \caption{Relative $L^2$ errors for reconstructing $q$ at different stages for Example \ref{5D_example} without noise.}
    \begin{tabular}{cccccccc} 
        \toprule
        Stage & 1 & 2 & 4 & 6 & 8 & 12 & 16\\
        \midrule
        $err(q)$ & $1.0\times10^{-1}$ & $3.9\times10^{-2}$ & $2.7\times10^{-2}$ & $2.4\times10^{-2}$& $1.9\times10^{-2}$ & $1.2\times10^{-2}$ & $9.7\times10^{-3}$\\
        \bottomrule
    \end{tabular}
    \label{5D_table_stage}
\end{table}

\begin{figure}[ht!]
    \centering
    \begin{subfigure}[b]{0.24\textwidth}
        \centering
        \includegraphics[width=\textwidth]{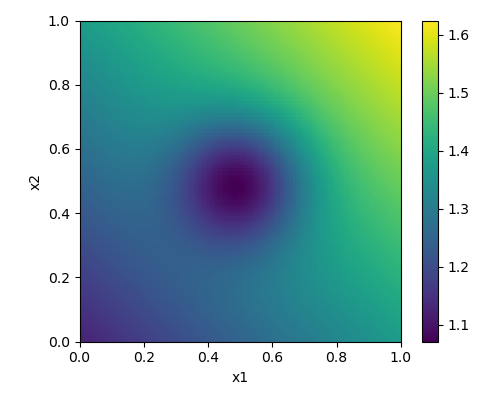}
        \caption{$q$}
    \end{subfigure}
    \hfill 
    \begin{subfigure}[b]{0.24\textwidth}
        \centering
        \includegraphics[width=\textwidth]{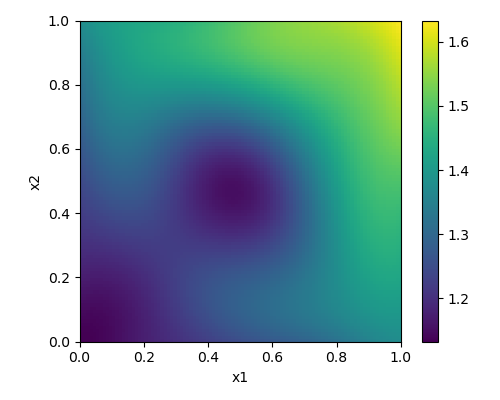}
        \caption{$q^{\dagger}$}
    \end{subfigure}
    \hfill
    \begin{subfigure}[b]{0.24\textwidth}
        \centering
        \includegraphics[width=\textwidth]{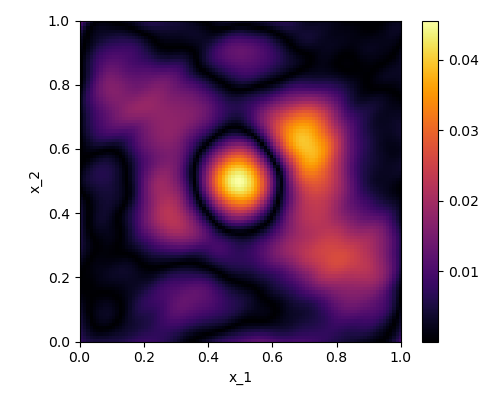}
        \caption{$|q-q^{\dagger}|$}
    \end{subfigure}
    \hfill
    \begin{subfigure}[b]{0.24\textwidth}
        \centering
        \includegraphics[width=\textwidth]{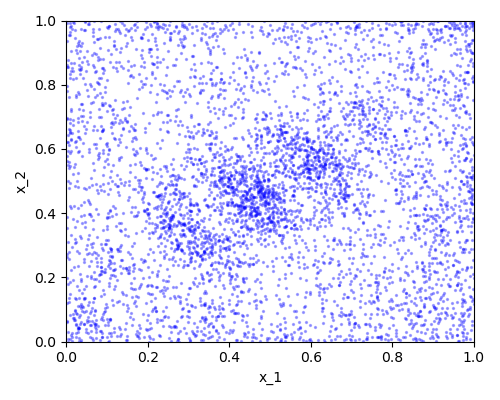}
        \caption{Sampling points for $q$}
    \end{subfigure}

    \centering
    \begin{subfigure}[b]{0.24\textwidth}
        \centering
        \includegraphics[width=\textwidth]{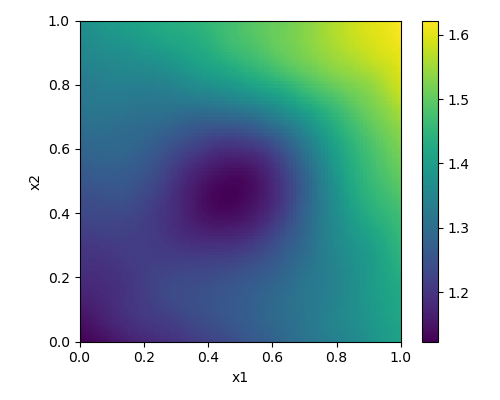}
        \caption{$q$}
    \end{subfigure}
    \hfill 
    \begin{subfigure}[b]{0.24\textwidth}
        \centering
        \includegraphics[width=\textwidth]{fig/5D/0_noise/true_q.png}
        \caption{$q^{\dagger}$}
    \end{subfigure}
    \hfill
    \begin{subfigure}[b]{0.24\textwidth}
        \centering
        \includegraphics[width=\textwidth]{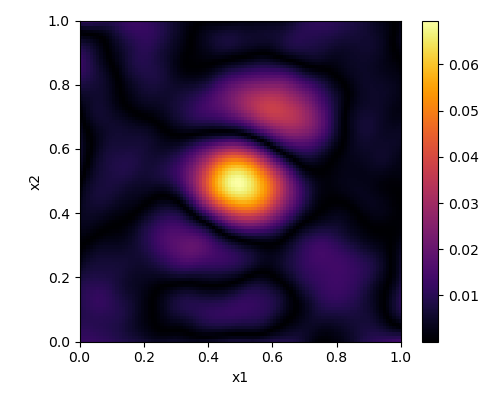}
        \caption{$|q-q^{\dagger}|$}
    \end{subfigure}
    \hfill
    \begin{subfigure}[b]{0.24\textwidth}
        \centering
        \includegraphics[width=\textwidth]{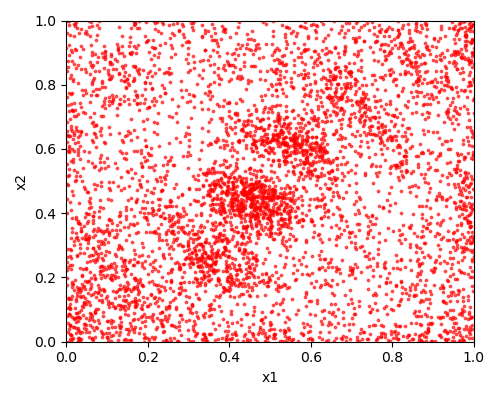}
        \caption{Sampling points for $q$}
    \end{subfigure}

       \centering
    \begin{subfigure}[b]{0.24\textwidth}
        \centering
        \includegraphics[width=\textwidth]{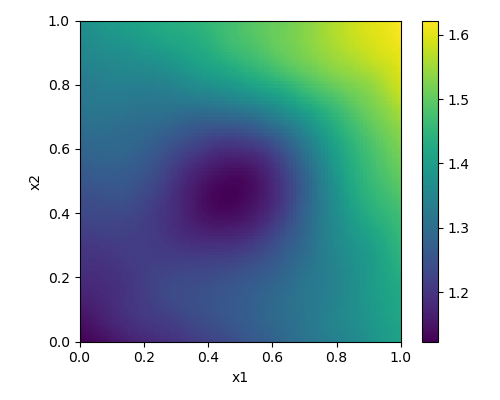}
        \caption{$q$}
    \end{subfigure}
    \hfill 
    \begin{subfigure}[b]{0.24\textwidth}
        \centering
        \includegraphics[width=\textwidth]{fig/5D/0_noise/true_q.png}
        \caption{$q^{\dagger}$}
    \end{subfigure}
    \hfill
    \begin{subfigure}[b]{0.24\textwidth}
        \centering
        \includegraphics[width=\textwidth]{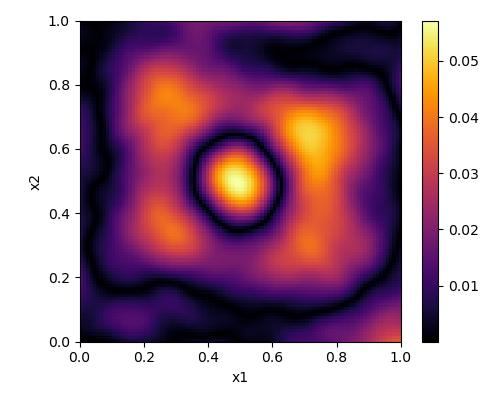}
        \caption{$|q-q^{\dagger}|$}
    \end{subfigure}
    \hfill
    \begin{subfigure}[b]{0.24\textwidth}
        \centering
        \includegraphics[width=\textwidth]{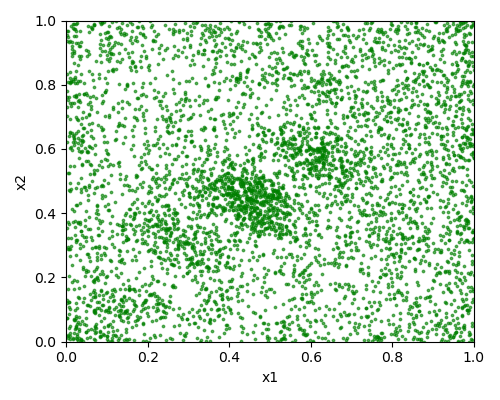}
        \caption{Sampling points for $q$}
    \end{subfigure}
    
    \caption{Reconstruction results for Example \ref{5D_example} with noise levels of 0\%, 10\%, and 20\% from top to bottom.}
    \label{fig:5D_error}
\end{figure}

\begin{table}[ht!]
    \centering
    \caption{Relative $L^2$ errors for reconstructing $q$ across different methods for Example \ref{5D_example}.}
    \begin{tabular}{ccccccccc}
        \toprule
        Method & \multicolumn{4}{c}{ALBC} & \multicolumn{3}{c}{Mixed PINN} \\
        
        \cmidrule(lr){2-5} \cmidrule(lr){6-8} 
        
        $\delta$ & 0\% & 10\% & 20\% & Time & 0\% & 10\% & Time \\
        \midrule
        
        Error & $9.70\times10^{-3}$ & $1.07\times10^{-2}$ & $1.57\times10^{-2}$& 56s & $5.78\times10^{-3}$ & $2.86\times10^{-2}$  & 962s \\
        
        \bottomrule
    \end{tabular}
    \label{5D_with_noise}
\end{table}

\begin{figure}[ht!]
    \centering
    \includegraphics[width=0.8\linewidth]{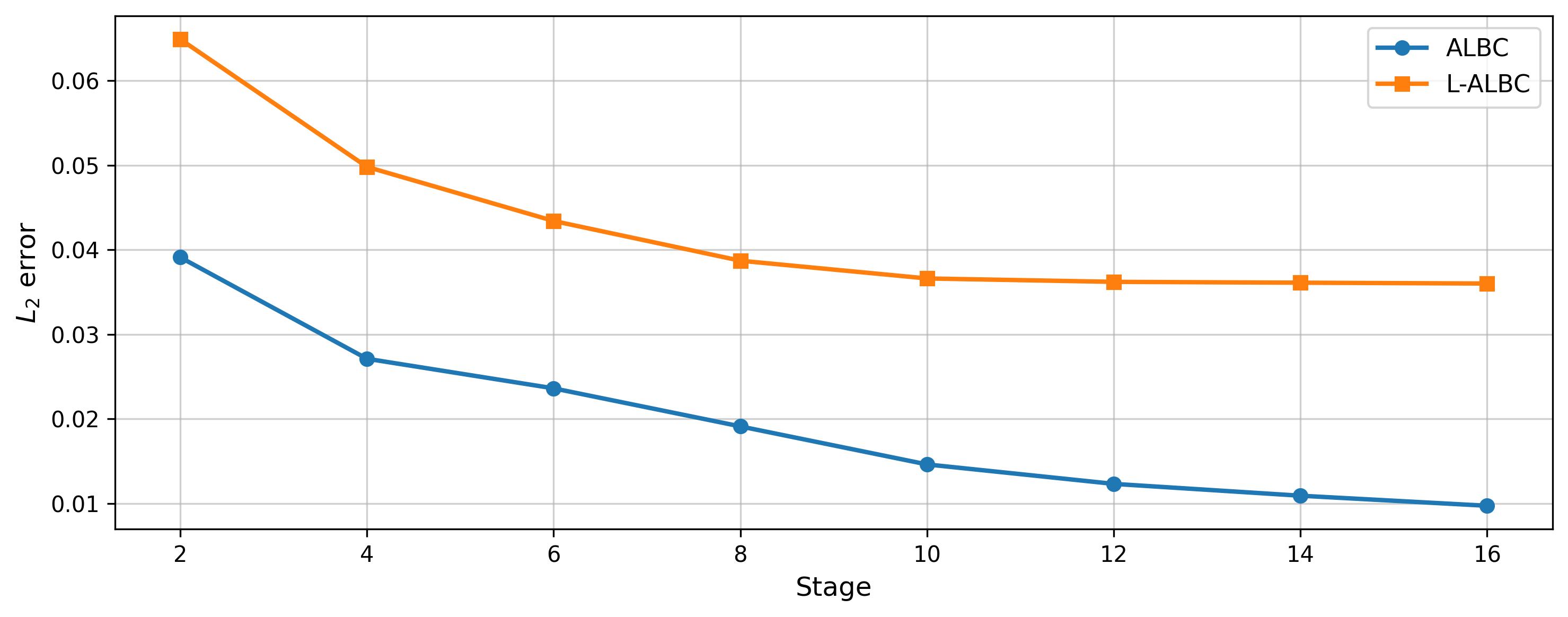}
    \caption{Comparison of the relative $L^2$ error across different stages for Example \ref{5D_example}. The plot compares our complete method against the variant without fine-tuning.}
    \label{fine-tuning}
\end{figure}

Table \ref{5D_table_stage} presents the relative $L^2$ errors for $q$ at different stages in the noise-free setting, while Figure \ref{fig:5D_error} illustrates the reconstructed $q$, the ground truth, the absolute error, and the distribution of sampling points under three noise levels. As observed, the error decreases steadily with successive stages. Furthermore, our method achieves an accurate reconstruction of $q$ across all tested noise levels, thereby confirming the effectiveness and robustness of the alternating strategy in high-dimensional settings. For a quantitative benchmark, Table \ref{5D_with_noise} compares our approach against the Mixed PINN baseline. Although Mixed PINNs yield a slightly lower error in the strictly noise-free case, our method demonstrates superior robustness under $10\%$ noise. More importantly, our approach requires only a fraction of the computational time ($56$s versus $962$s), thereby highlighting its exceptional scalability and noise resilience in high-dimensional inverse problems.

Finally, we utilize this example to verify the necessity of the joint fine-tuning phase in ALBC. As illustrated in Figure \ref{fine-tuning}, ALBC exhibits a significantly faster convergence in the $L^2$ error as the number of stages increases. Notably, even after 16 stages, ALBC continues to converge towards a lower error, whereas the error of L-ALBC plateaus.

\section{Conclusion}\label{sec1_conclusion}

In this paper, we introduced the Alternating Learning-Based Collocation (ALBC) method, which integrates alternating iterative strategies with sinusoidal-activated shallow neural networks for solving inverse elliptic problems. By employing shallow networks as adaptive basis generators within a collocation framework, the method alternately constructs dedicated basis sets for the state variable $u$ and the unknown parameter $\lambda$, effectively decoupling the non-convex joint optimization into tractable linear subproblems. Rigorous theoretical guarantees, including stability estimates and convergence analysis, have been established. Extensive numerical experiments across five representative inverse problems (source identification, potential reconstruction, diffusion coefficient recovery, electrical impedance tomography, and a five-dimensional problem) demonstrate that the ALBC method achieves high reconstruction accuracy with significantly lower computational costs compared to PINNs and classical RBF-based collocation methods. Looking forward, our future research will primarily focus on three directions. First, we plan to extend the ALBC framework to tackle more complex nonlinear inverse problems. Second, we aim to establish an observable numerical analysis theory. Finally, we intend to further refine the theoretical foundations of the adaptive sampling strategy to provide deeper mathematical insights into its efficiency and convergence properties.

\appendix

\appendix

\section{Auxiliary lemmas}

\begin{lemma}[Universal Approximation]\label{UOT}
\cite{hornik1991approximation}
     Suppose the activation function $\sigma$ is smooth, bounded, and non-constant. Then, for any target function $f$ belonging to a Sobolev space $W^{m,p}(\Omega)$, the sequence of functions generated by shallow neural networks converges to $f$ in the $W^{m,p}$ norm as the network width approaches infinity. In other words, the space of shallow neural networks is dense in $W^{m,p}(\Omega)$.
\end{lemma}

\begin{lemma}[Spectral Truncation]
\label{lemma:barron_bound}
Under Assumption 1 ($u^\dagger \in H^{p}(\Omega)$, $p > d/2 + 1$), define the band-limited hypothesis space
$$\mathcal{H}_{\omega_{\max}} := \Bigl\{ \textstyle\sum_{j=1}^{n} a_j \sin(\boldsymbol{\omega}_j \cdot \mathbf{x} + b_j) \;\Big|\; n \in \mathbb{N},\; a_j \in \mathbb{R},\; |\boldsymbol{\omega}_j| \le \omega_{\max} \Bigr\},$$
and let $u^\dagger_{<\omega_{\max}}$ denote the $L^2$-projection of $u^\dagger$ onto $\mathcal{H}_{\omega_{\max}}$ and $u^\dagger_{\mathrm{tail}}:=u^\dagger-u^\dagger_{<\omega_{\max}}$ its tail (and analogously $\lambda^\dagger_{<\omega_{\max}}$, $\lambda^\dagger_{\mathrm{tail}}$ for the parameter $\lambda^\dagger$). Then for any $0 \le s < p$,
$$\| u^\dagger_{\mathrm{tail}} \|_{H^{s}(\Omega)} \le \omega_{\max}^{-(p-s)} \| u^\dagger \|_{H^{p}(\Omega)}, \qquad \|\lambda^\dagger_{\mathrm{tail}}\|_{H^{s}(\Omega)} \le \omega_{\max}^{-(p-s)} \|\lambda^\dagger\|_{H^{p}(\Omega)}.$$
\end{lemma}

\begin{proof}[Proof of Lemma~\ref{lemma:barron_bound}]
Since $u^\dagger_{\mathrm{tail}}$ has Fourier support on $\{|\boldsymbol{\omega}|>\omega_{\max}\}$, for $0\le s<p$:
$$\|u^\dagger_{\mathrm{tail}}\|_{H^{s}(\Omega)}^2 = \int_{|\boldsymbol{\omega}|>\omega_{\max}} (1+|\boldsymbol{\omega}|^2)^{s} |\hat{u}^\dagger|^2\,d\boldsymbol{\omega} \le \omega_{\max}^{-2(p-s)}\|u^\dagger\|_{H^{p}(\Omega)}^2,$$
where the inequality uses $(1+|\boldsymbol{\omega}|^2)^{-(p-s)} \le \omega_{\max}^{-2(p-s)}$ on the integration domain; the bound for $\lambda^\dagger_{\mathrm{tail}}$ is identical.
\end{proof}

\section{Proof of Theorem \ref{thm:u}}
\begin{proof}
Let $e_u^{(k)} = u_k - u^\dagger \in H^s(\Omega)$ denote the absolute state approximation error at the $k$-th alternating stage. Since $u_k$ is constrained within a band-limited hypothesis space $\mathcal{H}_{\omega_{\max}}$, we decompose the exact solution $u^\dagger \in H^{p}(\Omega)$ into a spectral projection $u^\dagger_{<\omega_{\max}}$ onto $\mathcal{H}_{\omega_{\max}}$ and a high-frequency residual. By Lemma~\ref{lemma:barron_bound}, the truncation error satisfies:

$$\Vert u^\dagger - u^\dagger_{<\omega_{\max}} \Vert_{H^{s}(\Omega)}^2 = \mathcal{O}(\omega_{\max}^{-2(p-s)}).$$

Applying the triangle inequality, the total error is decoupled into the approximation error within the band-limited space and the spectral truncation error:

$$\Vert e_u^{(k)} \Vert_{H^{s}(\Omega)}^2 \le 2 \Vert u_k - u^\dagger_{<\omega_{\max}} \Vert_{H^{s}(\Omega)}^2 + 2 \Vert u^\dagger - u^\dagger_{<\omega_{\max}} \Vert_{H^{s}(\Omega)}^2.$$

For the band-limited component $u_k - u^\dagger_{<\omega_{\max}} \in \mathcal{H}_{\omega_{\max}}$, we apply the inverse inequality $\Vert v \Vert_{H^{s}} \le C_{inv}\omega_{\max}^{s} \Vert v \Vert_{L^2}$ for all $v \in \mathcal{H}_{\omega_{\max}}$, where $C_{inv}$ is a dimension-dependent constant. A further triangle inequality gives
$$\Vert u_k - u^\dagger_{<\omega_{\max}} \Vert_{L^2(\Omega)} \le \Vert u_k - u^\dagger \Vert_{L^2(\Omega)} + \Vert u^\dagger - u^\dagger_{<\omega_{\max}} \Vert_{L^2(\Omega)},$$
where the second term satisfies $\Vert u^\dagger - u^\dagger_{<\omega_{\max}} \Vert_{L^2} = \mathcal{O}(\omega_{\max}^{-p})$ by Lemma~\ref{lemma:barron_bound}. When multiplied by $C_{inv}^2\omega_{\max}^{2s}$, this contributes $\mathcal{O}(\omega_{\max}^{-2(p-s)})$ and is absorbed into the spectral truncation term. Setting $C_{stab} := 2C_{inv}$, we obtain:
$$\Vert e_u^{(k)} \Vert_{H^{s}(\Omega)}^2 \le C_{stab}^2\,\omega_{\max}^{2s}\, \Vert u_k - u^\dagger \Vert_{L^2(\Omega)}^2 + \mathcal{O}(\omega_{\max}^{-2(p-s)}).$$

To relate the $L^2$ state error to the continuous observation residual, we use the noise model $y = u^\dagger + \epsilon_{noise}$ defined on $\Omega$ with noise level $\delta_{noise}^2$. The triangle inequality yields:

$$\Vert u_k - u^\dagger \Vert_{L^2(\Omega)}^2 \le 2 \Vert u_k - y \Vert_{L^2(\Omega)}^2 + \mathcal{O}(\delta_{noise}^2).$$
Since $\Vert u_k - y \Vert_{L^2(\Omega)}^2$ equals the data-fidelity component of $\bar{\mathcal{R}}_u^{(k)}$ (cf.~\eqref{u_unnorm}), the non-negativity of the remaining boundary and physics terms gives $\Vert u_k - y \Vert_{L^2(\Omega)}^2 \le \bar{\mathcal{R}}_u^{(k)}$. Substituting back, we obtain the deterministic estimate:

$$\Vert u_k - u^\dagger \Vert_{H^{s}(\Omega)}^2 \le C_{stab}^2 \, \omega_{\max}^{2s} \big( \bar{\mathcal{R}}^{(k)}_u + \delta_{noise}^2 \big) + \mathcal{O}(\omega_{\max}^{-2(p-s)}).$$
\end{proof}

\section{Proof of Theorem \ref{thm:lambda}}

\begin{proof}
Let $e_{\lambda}^{(k)} = \lambda_k - \lambda^\dagger \in L^2(\Omega)$ and $e_u^{(k)} = u_k - u^\dagger \in H^2(\Omega)$ denote the approximation errors at the $k$-th alternating stage. Since $\mathcal{N}(u^\dagger, \lambda^\dagger) \equiv 0$, the physical residual $\mathcal{N}(u_k, \lambda_k)$ can be decoupled by adding and subtracting $\mathcal{N}(u^\dagger, \lambda_k)$:

$$\mathcal{A}_{u^\dagger} e_{\lambda}^{(k)} = \mathcal{N}(u_k, \lambda_k) - \underbrace{\big[ \mathcal{N}(u_k, \lambda_k) - \mathcal{N}(u^\dagger, \lambda_k) \big]}_{:= \mathcal{P}_u(e_u^{(k)}, \lambda_k)},$$
where $\mathcal{A}_{u^\dagger} e_\lambda := \mathcal{N}(u^\dagger, \lambda^\dagger + e_\lambda) - \mathcal{N}(u^\dagger, \lambda^\dagger) = \mathcal{N}_\lambda(e_\lambda)$ is the linearized parameter operator with $u^\dagger$ fixed. Applying $(a+b)^2 \le 2a^2 + 2b^2$:

$$\Vert \mathcal{A}_{u^\dagger} e_{\lambda}^{(k)} \Vert_{L^2(\Omega)}^2 \le 2\Vert \mathcal{N}(u_k, \lambda_k) \Vert_{L^2(\Omega)}^2 + 2\Vert \mathcal{P}_u(e_u^{(k)}, \lambda_k) \Vert_{L^2(\Omega)}^2.$$

Since $\mathcal{N}$ is second-order, $\mathcal{P}_u$ involves derivatives of $e_u^{(k)}$ up to order two. By Assumption~\ref{assump:uniform}, $\lambda_k$ is uniformly bounded in $W^{1,\infty}(\Omega)$, which guarantees a uniform Lipschitz constant $L_u < \infty$ depending only on $\mathcal{N}$ and $M_\lambda$:

$$\Vert \mathcal{P}_u(e_u^{(k)}, \lambda_k) \Vert_{L^2(\Omega)}^2 \le L_u \Vert e_u^{(k)} \Vert_{H^2(\Omega)}^2.$$

Subsequently, we analytically invert the linearized operator $\mathcal{A}_{u^{\dagger}}$ to isolate $e_{\lambda}$, which is customized for the specific structure of the inverse problem dictated by Assumption 2:
\begin{itemize}
    \item \textbf{Inverse Source Problem} ($\lambda = f$): The operator acts as the negative identity, $\mathcal{A}_{u^{\dagger}}e_{f} = -e_{f}$. This trivializes the inversion, yielding $\|e_{f}\|_{L^{2}(\Omega)} = \|\mathcal{A}_{u^{\dagger}}e_{f}\|_{L^{2}(\Omega)}$.
    \item \textbf{Inverse Potential Problem} ($\lambda = b$): The operator acts as a pointwise multiplier, $\mathcal{A}_{u^{\dagger}}e_{b} = u^{\dagger}e_{b}$. Imposing the strict physical non-degeneracy condition $|u^{\dagger}| \ge c_{0} > 0$ almost everywhere in $\Omega$, we obtain $\|e_{b}\|_{L^{2}(\Omega)} \le c_{0}^{-1} \|\mathcal{A}_{u^{\dagger}}e_{b}\|_{L^{2}(\Omega)}$.
\end{itemize}

Substituting the $L_u$ bound and applying the inversion inequalities established in the two cases above, we absorb all case-specific constants into a single stability constant $C_\lambda > 0$, yielding the continuous deterministic estimate:

$$\Vert e_{\lambda}^{(k)} \Vert_{L^2(\Omega)}^2 \le C_\lambda^2 \Big( \Vert \mathcal{N}(u_k, \lambda_k) \Vert_{L^2(\Omega)}^2 + L_u \Vert e_u^{(k)} \Vert_{H^2(\Omega)}^2 \Big).$$

Since $\bar{\mathcal{R}}_\lambda^{(k)}$ comprises the continuous $L^2$ physics residual $\Vert \mathcal{N}(u_k, \lambda_k) \Vert_{L^2(\Omega)}^2$ together with non-negative regularization terms (cf.~\eqref{lambda_unnorm}), we have $\bar{\mathcal{R}}_\lambda^{(k)} \ge \Vert \mathcal{N}(u_k, \lambda_k) \Vert_{L^2(\Omega)}^2$. We arrive at the final deterministic bound:

$$\Vert \lambda_k - \lambda^\dagger \Vert_{L^2(\Omega)}^2 \le C_\lambda^2 \Big( \bar{\mathcal{R}}_\lambda^{(k)} + L_u \Vert u_k - u^\dagger \Vert_{H^2(\Omega)}^2 \Big).$$
\end{proof}

\section{Proof of Theorem \ref{thm:convergence}}

\begin{proof}
As in the proof of Theorem~\ref{thm:lambda}, $e_u^{(k)}=u_k-u^\dagger$ and $e_\lambda^{(k)}=\lambda_k-\lambda^\dagger$ denote the stage-$k$ errors. We first derive a one-step contraction for the \emph{ideal} state and parameter losses, then convert each into a recursion for the \emph{actual} losses, and finally combine the two recursions into a single contracting functional.

The stage-$k$ state loss, evaluated at the state iterate $u_k$ with the parameter frozen at $\lambda_{k-1}$, reads
\begin{equation}
\bar{\mathcal{R}}_u^{(k)} = \Vert u_k - u_{data} \Vert_{L^2(\Omega)}^2 + \eta_1 \Vert \mathcal{B}u_k - g \Vert_{L^2(\partial\Omega)}^2 + \eta_2 \Vert \mathcal{N}(u_k, \lambda_{k-1}) \Vert_{L^2(\Omega)}^2.
\end{equation}
Here $u_{data}=u^\dagger+\epsilon_{noise}$ with $\|u_{data}-u^\dagger\|_{L^2(\Omega)}^2=\mathcal{O}(\delta_{noise}^2)$; for clarity we take $\mathcal{O}_l=\mathcal{I}$, a bounded and boundedly invertible $\mathcal{O}_l$ only rescaling the constants by $\|\mathcal{O}_l\|^2$ and its lower bound, the same conditioning already embedded in $C_{stab}$ of Theorem~\ref{thm:u}. Replacing $\lambda_{k-1}$ by $\lambda^\dagger$ defines the ideal counterpart
\begin{equation}
\bar{\mathcal{R}}_u^\dagger(u) = \Vert u - u_{data} \Vert_{L^2(\Omega)}^2 + \eta_1 \Vert \mathcal{B}u - g \Vert_{L^2(\partial\Omega)}^2 + \eta_2 \Vert \mathcal{N}(u, \lambda^\dagger) \Vert_{L^2(\Omega)}^2.
\end{equation}

With $u^\dagger_{<\omega_{\max}}$ and $u^\dagger_{\mathrm{tail}}$ as in Lemma~\ref{lemma:barron_bound}, the band-limited error is
\begin{equation*}
    \tilde{e}_u^{(k-1)} \;:=\; u^\dagger_{<\omega_{\max}} - u_{k-1} \;=\; -\,e_u^{(k-1)}-u^\dagger_{\mathrm{tail}} \;\in\; \mathcal{H}_{\omega_{\max}}.
\end{equation*}
Since the target $u^\dagger_{<\omega_{\max}}$ is fixed across stages, Lemma~\ref{UOT} produces, for any $\alpha_u\in(1/2,1]$, there exists an increment $\phi^\star$ of that width such that
\begin{equation}\label{eq:in_class_contraction}
    \|\tilde{e}_u^{(k-1)} - \phi^\star\|_{L^2(\Omega)}^2 \;\le\; (1-\alpha_u)\,\|\tilde{e}_u^{(k-1)}\|_{L^2(\Omega)}^2.
\end{equation}
Fix $\omega_{\max}$ and set $\delta\phi:=\phi^\star-\tilde{e}_u^{(k-1)}\in\mathcal{H}_{\omega_{\max}}$, the in-class residual left by the increment, which by~\eqref{eq:in_class_contraction} obeys $\|\delta\phi\|_{L^2}^2\le(1-\alpha_u)\|\tilde{e}_u^{(k-1)}\|_{L^2}^2$, so that $u_{k-1}+\phi^\star-u^\dagger=\delta\phi-u^\dagger_{\mathrm{tail}}$. Using $\mathcal{B}u^\dagger=g$ and $\mathcal{N}(u^\dagger,\lambda^\dagger)=0$ together with the linearity of $\mathcal{B}$ and of $\mathcal{N}(\cdot,\lambda^\dagger)$, the data, boundary and physics terms at $u_{k-1}+\phi^\star$ become, respectively,
\begin{equation*}
    \delta\phi-u^\dagger_{\mathrm{tail}}-\epsilon_{noise},\qquad \mathcal{B}(\delta\phi-u^\dagger_{\mathrm{tail}}),\qquad \mathcal{N}_u^{\lambda^\dagger}(\delta\phi-u^\dagger_{\mathrm{tail}}),
\end{equation*}
i.e.\ bounded operators of orders $0,1,2$ acting on $\delta\phi-u^\dagger_{\mathrm{tail}}$. We next split each by $(a+b)^2\le(1+\xi_0)a^2+(1+\tfrac1{\xi_0})b^2$ into the band-limited part $\delta\phi\in\mathcal{H}_{\omega_{\max}}$ and the tail $u^\dagger_{\mathrm{tail}}$, and bound them separately.

Writing $C_{\mathcal{B}}$ and $C_{\mathcal{N}}$ for the continuity constants of $\mathcal{B}$ and $\mathcal{N}_u^{\lambda^\dagger}$, applying the inverse inequalities $\|v\|_{H^1}\le C_{inv}\,\omega_{\max}\|v\|_{L^2}$ and $\|v\|_{H^2}\le C_{inv}\,\omega_{\max}^2\|v\|_{L^2}$ to $\delta\phi$, we have
\begin{equation*}
    \|u_{k-1}+\phi^\star-u_{data}\|_{L^2}^2\le(1+\xi_0)\|\delta\phi\|_{L^2}^2+\big(1+\tfrac1{\xi_0}\big)\,\mathcal{O}(\delta_{noise}^2+\omega_{\max}^{-2p}),
\end{equation*}
\begin{equation*}
\begin{aligned}
    \eta_1\|\mathcal{B}(u_{k-1}+\phi^\star)-g\|_{L^2(\partial\Omega)}^2
    &\le\eta_1(1+\xi_0)C_{\mathcal{B}}^2\omega_{\max}^2\|\delta\phi\|_{L^2}^2\\
    &\quad+\eta_1\big(1+\tfrac1{\xi_0}\big)C_{\mathcal{B}}^2\,\mathcal{O}(\omega_{\max}^{-2(p-1)}),
\end{aligned}
\end{equation*}
\begin{equation*}
\begin{aligned}
    \eta_2\|\mathcal{N}(u_{k-1}+\phi^\star,\lambda^\dagger)\|_{L^2(\Omega)}^2
    &\le\eta_2(1+\xi_0)C_{\mathcal{N}}^2\omega_{\max}^4\|\delta\phi\|_{L^2}^2\\
    &\quad+\eta_2\big(1+\tfrac1{\xi_0}\big)C_{\mathcal{N}}^2\,\mathcal{O}(\omega_{\max}^{-2(p-2)}).
\end{aligned}
\end{equation*}
Here the constant $C_{inv}$ is absorbed into $C_{\mathcal{B}}$ and $C_{\mathcal{N}}$. With $\kappa_u^2:=1+\eta_1C_{\mathcal{B}}^2\omega_{\max}^2+\eta_2C_{\mathcal{N}}^2\omega_{\max}^4$ and $C_{noise},C_{trunc}>0$ collecting the $(1+\tfrac1{\xi_0})$-weighted noise and tail factors, using Lemma~\ref{lemma:barron_bound}, we have 
\begin{equation}\label{eq:Ru_ineq1}
\begin{aligned}
    \bar{\mathcal{R}}_u^\dagger(u_{k-1}+\phi^\star)
    &\le(1+\xi_0)\kappa_u^2\,\|\delta\phi\|_{L^2}^2\\
    &\quad+C_{noise}\,\delta_{noise}^2+C_{\mathrm{trunc}}\,\omega_{\max}^{-2(p-2)}.
\end{aligned}
\end{equation}
Since the data term is one summand of $\bar{\mathcal{R}}_u^{(k-1)}$, $\|u_{k-1}-u_{data}\|_{L^2}^2\le\bar{\mathcal{R}}_u^{(k-1)}$; with $\tilde{e}_u^{(k-1)}=-(u_{k-1}-u_{data})-u^\dagger_{\mathrm{tail}}-\epsilon_{noise}$, the three-term bound $(a+b+c)^2\le3(a^2+b^2+c^2)$ yields
\begin{equation}\label{eq:eu_ineq}
    \|\tilde{e}_u^{(k-1)}\|_{L^2}^2\le 3\,\bar{\mathcal{R}}_u^{(k-1)}+\mathcal{O}(\delta_{noise}^2+\omega_{\max}^{-2(p-2)}).
\end{equation}
Combining \eqref{eq:in_class_contraction}, \eqref{eq:Ru_ineq1} and \eqref{eq:eu_ineq} together, we have
\begin{equation}\label{eq:UAT_corrected_0}
    \bar{\mathcal{R}}_u^\dagger(u_{k-1}+\phi^\star) \;\le\; 3(1+\xi_0)\kappa_u^2(1-\alpha_u)\,\bar{\mathcal{R}}_u^{(k-1)} \;+\; C_{noise}\,\delta_{noise}^2 \;+\; C_{\mathrm{trunc}}\,\omega_{\max}^{-2(p-2)}.
\end{equation}
For a given $\hat\alpha_u\in(1/2,1)$, we can chose $\alpha_u$ such that $3(1+\xi_0)\kappa_u^2(1-\alpha_u)<1-\hat\alpha_u$, leading to
\begin{equation}\label{eq:UAT_corrected}
    \bar{\mathcal{R}}_u^\dagger(u_{k-1}+\phi^\star) \;\le\; (1-\hat{\alpha}_u)\,\bar{\mathcal{R}}_u^{(k-1)} \;+\; C_{noise}\,\delta_{noise}^2 \;+\; C_{\mathrm{trunc}}\,\omega_{\max}^{-2(p-2)}.
\end{equation}
For simplicity, we henceforth denote $\hat\alpha_u$ as $\alpha_u$. The minimiser property of the state update gives $\bar{\mathcal{R}}_u^{(k)}\le\bar{\mathcal{R}}_u(u_{k-1}+\phi^\star;\lambda_{k-1})$, the actual loss with the parameter frozen at $\lambda_{k-1}$. Since $\mathcal{N}(u,\lambda_{k-1})=\mathcal{N}(u,\lambda^\dagger)+\mathcal{N}_\lambda(u,\lambda_{k-1}-\lambda^\dagger)$, Young's inequality $(a+b)^2\le(1+\xi_1)a^2+(1+\tfrac1{\xi_1})b^2$ relates the actual loss to its ideal counterpart,
\begin{equation*}
    \bar{\mathcal{R}}_u(u;\lambda_{k-1})\le(1+\xi_1)\,\bar{\mathcal{R}}_u^\dagger(u)+\eta_2\big(1+\tfrac1{\xi_1}\big)\|\mathcal{N}_\lambda(u,\lambda_{k-1}-\lambda^\dagger)\|_{L^2(\Omega)}^2.
\end{equation*}
Choosing $\xi_1=\alpha_u/4$ makes $(1+\xi_1)(1-\alpha_u)\le1-\tfrac{\alpha_u}{2}$, so substituting~\eqref{eq:UAT_corrected} at $u=u_{k-1}+\phi^\star$ into the previous display gives
\begin{equation}\label{eq:Rku_bound}
    \bar{\mathcal{R}}_u^{(k)}\le\Big(1-\tfrac{\alpha_u}{2}\Big)\bar{\mathcal{R}}_u^{(k-1)}+\eta_2\big(1+\tfrac1{\xi_1}\big)\big\|\mathcal{N}_\lambda(u_{k-1}+\phi^\star,\lambda_{k-1}-\lambda^\dagger)\big\|_{L^2(\Omega)}^2+\mathcal{O}\big(\delta_{noise}^2+\omega_{\max}^{-2(p-2)}\big).
\end{equation}
It remains to control the discrepancy $\mathcal{N}_{\lambda}$. Denote the Lipschitz constant of $\mathcal{N}_\lambda$ as $L_\lambda$, we have 
\begin{equation}\label{eq:disc2ek}
  \|\mathcal{N}_\lambda(u,e_\lambda^{(k-1)})\|_{L^2(\Omega)}^2\le L_\lambda^2 M_u^2\|e_\lambda^{(k-1)}\|_{L^2(\Omega)}^2.
\end{equation}
Using Theorem~\ref{thm:lambda} and Theorem~\ref{thm:u}, we have
\begin{equation}\label{eq:ek_bound}
    \|e_\lambda^{(k-1)}\|_{L^2}^2\le C_\lambda^2\big(\bar{\mathcal{R}}_\lambda^{(k-1)}+L_u\|e_u^{(k-1)}\|_{H^2}^2\big),
    \qquad
    \|e_u^{(k-1)}\|_{H^2}^2\le C_{stab}^2\omega_{\max}^4\big(\bar{\mathcal{R}}_u^{(k-1)}+\delta_{noise}^2\big)+\mathcal{O}(\omega_{\max}^{-2(p-2)}),
\end{equation}
denote $C_{cross,1} = C_{\xi_1}L_\lambda^2 M_u^2 C_\lambda^2$ with $C_{\xi_i}:=(1+1/\xi_i)$ and $\zeta_u=\mathcal{O}\big(\delta_{noise}^2+\omega_{\max}^{-2(p-2)}\big)$, insert \eqref{eq:ek_bound} and \eqref{eq:disc2ek} into \eqref{eq:Rku_bound}, we can obtain
\begin{equation}\label{eq:state_recursion}
  \bar{\mathcal{R}}_u^{(k)}\;\le\;\bigl(1-\tfrac{\alpha_u}{2}+\eta_2 C_{cross,1}L_u C_{stab}^2\omega_{\max}^4\bigr)\bar{\mathcal{R}}_u^{(k-1)}+\eta_2 C_{cross,1}\,\bar{\mathcal{R}}_\lambda^{(k-1)}+\zeta_u,
\end{equation}

The parameter update follows the same pattern. Define the ideal parameter loss
\begin{equation*}
    \bar{\mathcal{R}}_\lambda^\dagger(\lambda):=\|\mathcal{N}(u^\dagger,\lambda)\|_{L^2(\Omega)}^2=\|\mathcal{A}_{u^\dagger}(\lambda-\lambda^\dagger)\|_{L^2(\Omega)}^2,
\end{equation*}
where $\mathcal{A}_{u^\dagger}=\mathcal{N}_\lambda$ is the linearized parameter operator of Theorem~\ref{thm:lambda} and we used $\mathcal{N}(u^\dagger,\lambda^\dagger)=0$. With $\lambda^\dagger_{<\omega_{\max}}$, $\lambda^\dagger_{\mathrm{tail}}$ as in Lemma~\ref{lemma:barron_bound}, set the band-limited error $\tilde{e}_\lambda^{(k-1)} := \lambda^\dagger_{<\omega_{\max}} - \lambda_{k-1} = -e_\lambda^{(k-1)}-\lambda^\dagger_{\mathrm{tail}}\in\mathcal{H}_{\omega_{\max}}$. By the density in Lemma~\ref{UOT}, there exist a stage-independent width $n^\star_\lambda$ and an increment $\psi^\star$ such that
\begin{equation*}
    \|\tilde{e}_\lambda^{(k-1)} - \psi^\star\|_{L^2(\Omega)}^2 \;\le\; (1-\alpha_\lambda)\,\|\tilde{e}_\lambda^{(k-1)}\|_{L^2(\Omega)}^2.
\end{equation*}
Since $\lambda_{k-1}+\psi^\star-\lambda^\dagger=-(\tilde{e}_\lambda^{(k-1)}-\psi^\star)-\lambda^\dagger_{\mathrm{tail}}$, the inequality $(a+b)^2\le(1+\xi_2)a^2+(1+\tfrac1{\xi_2})b^2$ gives
\begin{equation}\label{eq:Rl_young}
    \bar{\mathcal{R}}_\lambda^\dagger(\lambda_{k-1}+\psi^\star)\le(1+\xi_2)\|\mathcal{A}_{u^\dagger}(\tilde{e}_\lambda^{(k-1)}-\psi^\star)\|_{L^2}^2+(1+\tfrac1{\xi_2})\|\mathcal{A}_{u^\dagger}\lambda^\dagger_{\mathrm{tail}}\|_{L^2}^2.
\end{equation}
By Assumption~\ref{ass1}, $u^\dagger\in H^p\hookrightarrow L^\infty(\Omega)$, so $\mathcal{A}_{u^\dagger}$ is bounded, $\|\mathcal{A}_{u^\dagger}v\|_{L^2}\le L_\lambda\|v\|_{L^2}$, while Theorem~\ref{thm:lambda} gives the reverse stability bound $\|v\|_{L^2}\le C_\lambda\|\mathcal{A}_{u^\dagger}v\|_{L^2}$. With $\tilde{e}_\lambda^{(k-1)}=-e_\lambda^{(k-1)}-\lambda^\dagger_{\mathrm{tail}}$ and $\|\mathcal{A}_{u^\dagger}e_\lambda^{(k-1)}\|_{L^2}^2=\bar{\mathcal{R}}_\lambda^\dagger(\lambda_{k-1})$, the two terms of the split obey
\begin{align*}
    \|\mathcal{A}_{u^\dagger}(\tilde{e}_\lambda^{(k-1)}-\psi^\star)\|_{L^2}^2
    &\le L_\lambda^2\|\tilde{e}_\lambda^{(k-1)}-\psi^\star\|_{L^2}^2
    \le L_\lambda^2(1-\alpha_\lambda)\|\tilde{e}_\lambda^{(k-1)}\|_{L^2}^2,\\
    \|\tilde{e}_\lambda^{(k-1)}\|_{L^2}^2
    &\le C_\lambda^2\|\mathcal{A}_{u^\dagger}\tilde{e}_\lambda^{(k-1)}\|_{L^2}^2
    \le 2C_\lambda^2\big(\bar{\mathcal{R}}_\lambda^\dagger(\lambda_{k-1})+\|\mathcal{A}_{u^\dagger}\lambda^\dagger_{\mathrm{tail}}\|_{L^2}^2\big),
\end{align*}
the four inequalities using, in order, the boundedness of $\mathcal{A}_{u^\dagger}$, the contraction, the stability bound, and $(a+b)^2\le2(a^2+b^2)$. The tail term is controlled by the boundedness of $\mathcal{A}_{u^\dagger}$ and Lemma~\ref{lemma:barron_bound} at $s=0$,
\begin{equation*}
    \|\mathcal{A}_{u^\dagger}\lambda^\dagger_{\mathrm{tail}}\|_{L^2}^2\le L_\lambda^2\|\lambda^\dagger_{\mathrm{tail}}\|_{L^2}^2\le L_\lambda^2\,\omega_{\max}^{-2p}\|\lambda^\dagger\|_{H^p(\Omega)}^2.
\end{equation*}
Substituting all of these into \eqref{eq:Rl_young} gives
\begin{equation*}
    \bar{\mathcal{R}}_\lambda^\dagger(\lambda_{k-1}+\psi^\star)\le \kappa_\lambda^2(1-\alpha_\lambda)\,\bar{\mathcal{R}}_\lambda^\dagger(\lambda_{k-1})+C'_{\mathrm{trunc}}\,\omega_{\max}^{-2p}\|\lambda^\dagger\|_{H^p}^2,\qquad \kappa_\lambda^2:=2(1+\xi_2)L_\lambda^2C_\lambda^2.
\end{equation*}
Similar to \eqref{eq:UAT_corrected}, by redefining $\alpha_{\lambda}$ properly the above inequality can be written as
\begin{equation}\label{eq:lambda_UAT_corrected}
    \bar{\mathcal{R}}_\lambda^\dagger(\lambda_{k-1} + \psi^\star) \;\le\; (1-\alpha_\lambda)\,\bar{\mathcal{R}}_\lambda^\dagger(\lambda_{k-1}) \;+\; C'_{\mathrm{trunc}}\,\omega_{\max}^{-2p}\,\|\lambda^\dagger\|^2_{H^p(\Omega)},
\end{equation}
with $C'_{\mathrm{trunc}}>0$ the parameter-side analog of $C_{\mathrm{trunc}}$, collecting $L_\lambda$, $C_\lambda$, the Young factor and the Lemma~\ref{lemma:barron_bound} tail constant. Defining the state perturbation operator $\mathcal{P}_u^{(j)}=\mathcal{N}(u_j,\,\cdot\,)-\mathcal{N}(u^\dagger,\,\cdot\,)$, the Lipschitz dependence of $\mathcal{N}$ on $u$ (Theorem~\ref{thm:lambda}, constant $L_u$) together with the second bound of~\eqref{eq:ek_bound} gives
\begin{equation}\label{eq:Pu_bound}
    \|\mathcal{P}_u^{(j)}\|_{L^2}^2\le L_u\|e_u^{(j)}\|_{H^2}^2\le L_u C_{stab}^2\omega_{\max}^4\,\bar{\mathcal{R}}_u^{(j)}+\zeta_u',
\end{equation}
where $\zeta_u':=L_u C_{stab}^2\omega_{\max}^4\,\delta_{noise}^2+L_u\,\mathcal{O}(\omega_{\max}^{-2(p-2)})$ collects the noise and tail contributions inherited from the $H^2$-stability bound.

Apply two Young splits, each using \eqref{eq:Pu_bound}, to convert the two endpoints of~\eqref{eq:lambda_UAT_corrected}. The minimiser property $\bar{\mathcal{R}}_\lambda^{(k)}\le\bar{\mathcal{R}}_\lambda(\lambda_{k-1}+\psi^\star;u_k)$ combined with $\mathcal{N}(u_k,\lambda)=\mathcal{N}(u^\dagger,\lambda)+\mathcal{P}_u^{(k)}\lambda$ yields
\begin{equation}\label{eq:lhs_split}
    \bar{\mathcal{R}}_\lambda^{(k)}\le(1+\xi_3)\,\bar{\mathcal{R}}_\lambda^\dagger(\lambda_{k-1}+\psi^\star)+C_{\xi_3}\bigl(L_u C_{stab}^2\omega_{\max}^4\,\bar{\mathcal{R}}_u^{(k)}+\zeta_u'\bigr),
\end{equation}
and the same split with $\mathcal{P}_u^{(k-1)}$ acting through $u_{k-1}$ yields
\begin{equation}\label{eq:rhs_split}
    \bar{\mathcal{R}}_\lambda^\dagger(\lambda_{k-1})\le(1+\xi_4)\,\bar{\mathcal{R}}_\lambda^{(k-1)}+C_{\xi_4}\bigl(L_u C_{stab}^2\omega_{\max}^4\,\bar{\mathcal{R}}_u^{(k-1)}+\zeta_u'\bigr).
\end{equation}
Substitute \eqref{eq:rhs_split} into the right-hand side of \eqref{eq:lambda_UAT_corrected} and then plug the resulting bound for $\bar{\mathcal{R}}_\lambda^\dagger(\lambda_{k-1}+\psi^\star)$ into \eqref{eq:lhs_split}. Choose $\xi_4>0$ so that $(1+\xi_4)(1-\alpha_\lambda)\le1-\tfrac{3\alpha_\lambda}{4}$ and then $\xi_3>0$ so that $(1+\xi_3)(1-\tfrac{3\alpha_\lambda}{4})\le1-\tfrac{\alpha_\lambda}{2}$. This gives
\begin{equation}\label{eq:lambda_recursion}
    \bar{\mathcal{R}}_\lambda^{(k)}\;\le\;\bigl(1-\tfrac{\alpha_\lambda}{2}\bigr)\bar{\mathcal{R}}_\lambda^{(k-1)}+C_{cross,2}\,\bar{\mathcal{R}}_u^{(k)}+C_{cross,2}'\,\bar{\mathcal{R}}_u^{(k-1)}+\zeta_\lambda,
\end{equation}
with $C_{cross,2}:=C_{\xi_3}L_u C_{stab}^2\omega_{\max}^4$, $C_{cross,2}':=(1+\xi_3)(1-\alpha_\lambda)C_{\xi_4}L_u C_{stab}^2\omega_{\max}^4$ and $\zeta_\lambda:=\bigl(C_{\xi_3}+(1+\xi_3)(1-\alpha_\lambda)C_{\xi_4}\bigr)\zeta_u'+(1+\xi_3)C'_{\mathrm{trunc}}\,\omega_{\max}^{-2p}\|\lambda^\dagger\|_{H^p}^2$.

Substituting \eqref{eq:state_recursion} into the $C_{cross,2}\bar{\mathcal{R}}_u^{(k)}$ term of \eqref{eq:lambda_recursion} and using $L_u C_{stab}^2\omega_{\max}^4\le C_{cross,2}$ in the resulting state diagonal,
\begin{equation}\label{eq:lambda_after_sub}
    \bar{\mathcal{R}}_\lambda^{(k)}\le\Big(1-\tfrac{\alpha_\lambda}{2}+\eta_2 C_{cross,1}C_{cross,2}\Big)\bar{\mathcal{R}}_\lambda^{(k-1)}+\tilde{C}_{cross,2}\,\bar{\mathcal{R}}_u^{(k-1)}+\zeta_\lambda+C_{cross,2}\zeta_u,
\end{equation}
where $\tilde{C}_{cross,2}:=C_{cross,2}\big(1-\tfrac{\alpha_u}{2}+\eta_2 C_{cross,1}C_{cross,2}\big)+C_{cross,2}'$. Applying the same diagonal over-bound to \eqref{eq:state_recursion} itself and pairing it with \eqref{eq:lambda_after_sub} gives the coupled system
$$
\begin{pmatrix} \bar{\mathcal{R}}_u^{(k)} \\ \bar{\mathcal{R}}_\lambda^{(k)} \end{pmatrix} \;\leq\; \mathbf{M} \begin{pmatrix} \bar{\mathcal{R}}_u^{(k-1)} \\ \bar{\mathcal{R}}_\lambda^{(k-1)} \end{pmatrix} + \boldsymbol{\zeta},
\qquad
\boldsymbol{\zeta} = \big(\zeta_u,\ \zeta_\lambda+C_{cross,2}\zeta_u\big)^{\mathsf{T}},
$$
$$
\mathbf{M}:=\begin{pmatrix} M_{11} & M_{12}\\[2pt] M_{21} & M_{22}\end{pmatrix},\quad
\begin{aligned}
&M_{11}:=1-\tfrac{\alpha_u}{2}+\eta_2 C_{cross,1}C_{cross,2}, &&M_{12}:=\eta_2 C_{cross,1},\\
&M_{22}:=1-\tfrac{\alpha_\lambda}{2}+\eta_2 C_{cross,1}C_{cross,2}, &&M_{21}:=\tilde{C}_{cross,2},
\end{aligned}
$$
Since $M_{12},M_{21}>0$, the discriminant $(M_{11}-M_{22})^2+4M_{12}M_{21}>0$, so the eigenvalues of $\mathbf{M}$ are real and its spectral radius equals the larger one:
$$\rho(\mathbf{M})=\tfrac{M_{11}+M_{22}}{2}+\tfrac{1}{2}\sqrt{(M_{11}-M_{22})^2+4 M_{12}M_{21}}\le\max(M_{11},M_{22})+\sqrt{M_{12}M_{21}},$$
the inequality using $\sqrt{x+y}\le\sqrt{x}+\sqrt{y}$ on $x=(M_{11}-M_{22})^2$ and $y=4M_{12}M_{21}$. Set $\alpha:=\min(\alpha_u,\alpha_\lambda)\in(1/2,1)$, so that $\max(M_{11},M_{22})=1-\alpha/2+\eta_2 C_{cross,1}C_{cross,2}$. Restrict the physics weight $\eta_2$ by
\begin{equation}\label{eq:eta_restriction}
    \eta_2\le\frac{\alpha^2}{64\,C_{cross,1}\tilde{C}_{cross,2}}.
\end{equation}
The definition of $\tilde{C}_{cross,2}$ with $\alpha_u<1$ gives $\tilde{C}_{cross,2}\ge C_{cross,2}(1-\alpha_u/2)\ge C_{cross,2}/2$; substituting this and \eqref{eq:eta_restriction},
$$\eta_2 C_{cross,1}C_{cross,2}\le\frac{\alpha^2 C_{cross,2}}{64\,\tilde{C}_{cross,2}}\le\frac{\alpha^2}{32},\qquad \sqrt{M_{12}M_{21}}=\sqrt{\eta_2 C_{cross,1}\tilde{C}_{cross,2}}\le\frac{\alpha}{8}.$$
Insert both bounds into the spectral-radius estimate:
$$\rho(\mathbf{M})\le 1-\tfrac{\alpha}{2}+\tfrac{\alpha^2}{32}+\tfrac{\alpha}{8}\le 1-\tfrac{11\alpha}{32}=:\rho,$$
and $\alpha>1/2$ then gives $\rho<1-\alpha/4<7/8$.

All entries of $\mathbf{M}$ are non-negative, and $M_{12},M_{21}>0$ makes $\mathbf{M}$ irreducible; the Perron--Frobenius theorem then produces a left eigenvector $w=(\kappa_1,\kappa_2)^{\mathsf{T}}$ with $\kappa_1,\kappa_2>0$ and $w^{\mathsf{T}}\mathbf{M}=\rho(\mathbf{M})\,w^{\mathsf{T}}\le\rho\,w^{\mathsf{T}}$. Define $\mathcal{J}^{(k)}:=\kappa_1\bar{\mathcal{R}}_u^{(k)}+\kappa_2\bar{\mathcal{R}}_\lambda^{(k)}=w^{\mathsf{T}}\bigl(\bar{\mathcal{R}}_u^{(k)},\bar{\mathcal{R}}_\lambda^{(k)}\bigr)^{\mathsf{T}}$. Taking the inner product of the coupled system with $w$,
$$\mathcal{J}^{(k)}\le(w^{\mathsf{T}}\mathbf{M})\bigl(\bar{\mathcal{R}}_u^{(k-1)},\bar{\mathcal{R}}_\lambda^{(k-1)}\bigr)^{\mathsf{T}}+w^{\mathsf{T}}\boldsymbol{\zeta}\le\rho\,\mathcal{J}^{(k-1)}+\zeta,\qquad \zeta:=\kappa_1\zeta_u+\kappa_2(\zeta_\lambda+C_{cross,2}\zeta_u).$$
Iterating from $\mathcal{J}^{(0)}$,
\begin{equation*}
    \mathcal{J}^{(k)}\le\rho^{k}\mathcal{J}^{(0)}+\Big(\sum_{j=0}^{k-1}\rho^{j}\Big)\zeta=\rho^{k}\mathcal{J}^{(0)}+\frac{1-\rho^{k}}{1-\rho}\,\zeta,
\end{equation*}
To extract per-component asymptotics, observe that $\rho(\mathbf{M})<1$ guarantees $(I-\mathbf{M})$ is invertible with non-negative inverse, so the coupled system at $k\to\infty$ yields $\bar{\boldsymbol{\mathcal{R}}}^{(\infty)}\le(I-\mathbf{M})^{-1}\boldsymbol{\zeta}$ entrywise.

Under \eqref{eq:eta_restriction} together with $\tilde{C}_{cross,2}\ge C_{cross,2}/2$, we have $\eta_2 C_{cross,1}C_{cross,2}\le\alpha^2/32$, hence
$$1-M_{11},\;1-M_{22}\;\ge\;\tfrac{\alpha}{2}-\tfrac{\alpha^2}{32}\;\ge\;\tfrac{\alpha}{4},\qquad M_{12}M_{21}=\eta_2 C_{cross,1}\tilde{C}_{cross,2}\le\tfrac{\alpha^2}{64}.$$
Therefore
$$\det(I-\mathbf{M})=(1-M_{11})(1-M_{22})-M_{12}M_{21}\ge\tfrac{\alpha^2}{16}-\tfrac{\alpha^2}{64}=\tfrac{3\alpha^2}{64},$$
and the matrix inversion gives
\begin{align}
\bar{\mathcal{R}}_u^{(\infty)}&\le\tfrac{1}{\det(I-\mathbf{M})}\bigl[(1-M_{22})\zeta_u+M_{12}(\zeta_\lambda+C_{cross,2}\zeta_u)\bigr],\label{eq:Ru_infty_matrix}\\
\bar{\mathcal{R}}_\lambda^{(\infty)}&\le\tfrac{1}{\det(I-\mathbf{M})}\bigl[M_{21}\zeta_u+(1-M_{11})(\zeta_\lambda+C_{cross,2}\zeta_u)\bigr].\label{eq:Rlambda_infty_matrix}
\end{align}

Substitute $\zeta_u = \mathcal{O}(\delta_{noise}^2+\omega_{\max}^{-2(p-2)})$ and $\zeta_\lambda=\mathcal{O}(\omega_{\max}^4\delta_{noise}^2+\omega_{\max}^{-2(p-2)})$ into the numerator of \eqref{eq:Ru_infty_matrix}. The first term $(1-M_{22})\zeta_u$ is directly of order $\mathcal{O}(\delta_{noise}^2+\omega_{\max}^{-2(p-2)})$ since $1-M_{22}=\mathcal{O}(1)$. The second term $M_{12}\zeta_\lambda$ pairs an $\mathcal{O}(\omega_{\max}^{-4})$ factor from $M_{12}$ with the $\mathcal{O}(\omega_{\max}^4\delta_{noise}^2+\omega_{\max}^{-2(p-2)})$ bound for $\zeta_\lambda$, giving $\mathcal{O}(\delta_{noise}^2+\omega_{\max}^{-2p})$. The third term $M_{12}C_{cross,2}\zeta_u$ pairs $M_{12}C_{cross,2}=\mathcal{O}(1)$ with $\zeta_u$, giving $\mathcal{O}(\delta_{noise}^2+\omega_{\max}^{-2(p-2)})$. Hence
\begin{equation}\label{eq:Ru_infty_clean}
    \bar{\mathcal{R}}_u^{(\infty)}=\mathcal{O}\bigl(\delta_{noise}^2+\omega_{\max}^{-2(p-2)}\bigr).
\end{equation}

For \eqref{eq:Rlambda_infty_matrix}, $M_{21}\zeta_u$ pairs an $\mathcal{O}(\omega_{\max}^4)$ factor from $M_{21}$ with $\zeta_u$, giving
$$M_{21}\zeta_u=\mathcal{O}\bigl(\omega_{\max}^4\delta_{noise}^2+\omega_{\max}^{-2(p-4)}\bigr),$$
and $(1-M_{11})(\zeta_\lambda+C_{cross,2}\zeta_u)$ is of the same order. Hence
\begin{equation}\label{eq:Rlambda_infty_clean}
    \bar{\mathcal{R}}_\lambda^{(\infty)}=\mathcal{O}\bigl(\omega_{\max}^4\delta_{noise}^2+\omega_{\max}^{-2(p-4)}\bigr).
\end{equation}
\end{proof}

\section{Proof of Theorem \ref{thm:balance}}
\begin{proof}
Using the stability estimate of Theorem~\ref{thm:u} with $s=0$, we have
$$
\|u_k-u^\dagger\|_{L^2}^2\le C_{stab}^2\big(\bar{\mathcal{R}}_u^{(k)}+\delta_{noise}^2\big)+\mathcal{O}(\omega_{\max}^{-2p}).
$$
According to Theorem~\ref{thm:convergence}, as $k\to\infty$, $\bar{\mathcal{R}}_u^{(k)}=\mathcal{O}(\delta_{noise}^2+\omega_{\max}^{-2(p-2)})$, so in the limit the state error satisfies
$$
\|u_{\infty}-u^\dagger\|_{L^2}^2=\mathcal{O}\big(\delta_{noise}^2+\omega_{\max}^{-2(p-2)}\big),
$$
For the parameter field, Theorem~\ref{thm:lambda} bounds $\|\lambda_k-\lambda^\dagger\|_{L^2}^2$ by $\bar{\mathcal{R}}_\lambda^{(k)}+L_u\|u_k-u^\dagger\|_{H^2}^2$. Applying Theorem~\ref{thm:u} with $s=2$ to the state term and again letting $k\to\infty$ gives
$$
\|\lambda_{\infty}-\lambda^\dagger\|_{L^2}^2=\mathcal{O}\big(\omega_{\max}^4\,\delta_{noise}^2+\omega_{\max}^{-2(p-4)}\big).
$$
Denote $g(\omega_{\max})=\omega_{\max}^4\,\delta_{noise}^2+\omega_{\max}^{-2(p-4)}$. Solving $g'(\omega_{\max})=0$, we have that by selecting 
$$
\omega_{\max}^\star\sim\delta_{noise}^{-1/(p-2)},
$$
$\|\lambda_k-\lambda^\dagger\|_{L^2}^2$ and $\|u_k-u^\dagger\|_{L^2}$ achieve the optimal order (as $k\to\infty$):
$$
\|u_{\infty}-u^\dagger\|_{L^2}^2=\mathcal{O}(\delta_{noise}^2),
\qquad
\|\lambda_{\infty}-\lambda^\dagger\|_{L^2}^2=g(\omega_{\max}^\star)=\mathcal{O}\big(\delta_{noise}^{\,2(p-4)/(p-2)}\big).
$$
\end{proof}




\bibliographystyle{plain}
\bibliography{sample}

\end{document}